\documentclass[12pt]{amsart}
\usepackage{amsfonts, amssymb, latexsym, hyperref, xcolor, dynkin-diagrams, tikz,tikz-cd}
\usepackage{ytableau}

\newtheorem{theorem}{Theorem}[section]
\newtheorem{proposition}[theorem]{Proposition}
\newtheorem{lemma}[theorem]{Lemma}

\newtheorem{claim}[theorem]{Claim}
\newtheorem*{claim*}{Claim}
\newtheorem{corollary}[theorem]{Corollary}
\newtheorem{Main Conjecture}[theorem]{Main Conjecture}

\theoremstyle{definition}
\newtheorem{definition}[theorem]{Definition}
\theoremstyle{remark}

\newtheorem{example}[theorem]{Example}

\theoremstyle{plain}

\newcommand{\seg}[2]{\mathrm{seg}_{#1}^{#2}}

\newcommand{\qkt}[1]{\emph{(QK#1)}}
\newcommand{\rmin}[2]{{\sf rmin}_{#2}(#1)}
\newcommand{\rmax}[2]{{\sf rmax}_{#2}(#1)}
\newcommand{\flex}[2]{{\sf flex}_{#2}(#1)}

\newcommand{\LL}{\mathcal{L}}

\newcommand{\cellsize}{12}
\newlength{\cellsz} 
\newsavebox{\cell}
\sbox{\cell}{\begin{picture}(\cellsize,\cellsize)
\put(0,0){\line(1,0){\cellsize}}
\put(0,0){\line(0,1){\cellsize}}
\put(\cellsize,0){\line(0,1){\cellsize}}
\put(0,\cellsize){\line(1,0){\cellsize}}
\end{picture}}
\newcommand\cellify[1]{\def\thearg{#1}\def\nothing{}%
\ifx\thearg\nothing
\vrule width0pt height\cellsz depth0pt\else
\hbox to 0pt{\usebox{\cell} \hss}\fi%
\vbox to \cellsz{
\vss
\hbox to \cellsz{\hss$#1$\hss}
\vss}}
\newcommand\tableau[1]{\vtop{\let\\\cr
\baselineskip -16000pt \lineskiplimit 16000pt \lineskip 0pt
\ialign{&\cellify{##}\cr#1\crcr}}}

\newcommand{\excise}[1]{}

\begin{document}
\pagestyle{plain}
\title{Multiplicity-free key polynomials}
\author{Reuven Hodges}
\author{Alexander Yong}
\address{Dept.~of Mathematics, U.~Illinois at Urbana-Champaign, Urbana, IL 61801, USA} 
\email{rhodges@illinois.edu, ayong@illinois.edu}
\date{July 17, 2020}
\maketitle

\begin{abstract}
The \emph{key polynomials}, defined by A.~Lascoux--M.-P.~Sch\"utzenberger, are characters for the Demazure modules
of type $A$. We classify multiplicity-free key polynomials. The proof uses two combinatorial models for key polynomials.
The first is due to A.~Kohnert. The second is by S.~Assaf--D.~Searles, in terms of
\emph{quasi-key polynomials}. Our argument proves a sufficient condition 
for a quasi-key polynomial to be multiplicity-free. \end{abstract}

\section{Introduction}\label{sec:1}

 This is the companion paper to \cite{Hodges.Yong}. That work studies multiplicity-freeness of
 key polynomials in the context of \emph{spherical Schubert geometry}. We refer the reader to it for 
 additional motivation and references about the main result, Theorem~\ref{thm:mfKey}.
 
Let ${\sf Pol}_n={\mathbb Z}[x_1,\ldots,x_n]$. The \emph{Demazure operator} $\pi_j:{\sf Pol}_n\to {\sf Pol}_n$
is defined by
\[f\mapsto \frac{x_j f - x_{j+1}s_jf}{x_j-x_{j+1}}, \text{
\ where \  ${s_j}f:=f(x_1,\ldots,x_{j+1},x_j,\ldots,x_n)$.}\]

A \emph{weak composition} of length $n$ is $\alpha=(\alpha_1,\ldots,
\alpha_{n})\in {\mathbb Z}_{\geq 0}^n$.
Let ${\sf Comp}_n$ be the set of such $\alpha$.
If $\alpha \in {\sf Comp}_n$ is weakly decreasing, the \emph{key polynomial} $\kappa_{\alpha}$ is
$x^{\alpha}:=x_1^{\alpha_1}\cdots x_n^{\alpha_n}$. 
Otherwise, 
\[
\kappa_{\alpha}=\pi_j(\kappa_{\widehat \alpha}) \mbox{\ where $\widehat\alpha=(\alpha_1,\ldots,\alpha_{j+1},\alpha_j,\ldots,\alpha_n)$ and
$\alpha_{j+1}>\alpha_{j}$.}
\]
The key polynomials for $\alpha \in {\sf Comp}_n$
form a ${\mathbb Z}$-basis of ${\mathbb Z}[x_1,\ldots,x_n]$; see work of V.~Reiner--M.~Shimozono \cite{Reiner.Shimozono} and of A.~Lascoux \cite{Lascoux:polynomials} (and references therein) for more on $\kappa_{\alpha}$. In \cite[Section~4.4]{Hodges.Yong} we use the fact that $\kappa_{\alpha}$ is the character of a \emph{Demazure module} of $B\subset GL_n$
\cite{Reiner.Shimozono, Ion, Mason}. We do not need this in the present paper, which is entirely combinatorial.

Let ${\sf Comp} := \bigcup_{n=1}^{\infty}  {\sf Comp}_n$.  For $\alpha=(\alpha_1,\ldots,
\alpha_{\ell}), \beta=(\beta_1,\ldots, \beta_{k}) \in {\sf Comp}$, $\alpha$ \emph{contains the composition pattern} $\beta$ 
if there exists integers $j_1 < j_2 < \cdots < j_k$ that satisfy:
\begin{itemize}
\item $\alpha_{j_s} \leq \alpha_{j_t}$ if and only if $\beta_{s} \leq \beta_{t}$, 
\item $|\alpha_{j_s} - \alpha_{j_t}| \geq |\beta_{s} - \beta_{t}|$.
\end{itemize}
If $\alpha$ does not contain $\beta$, $\alpha$ \emph{avoids} $\beta$. This is a recapitulation of \cite[Definition~4.8]{Hodges.Yong}. 
Let \[ {\sf KM} = \{ (0,1,2), (0,0,2,2), (0,0,2,1), (1,0,3,2), (1,0,2,2) \}. \] 
Define ${\overline {\sf KM}}_n$ to be those $\alpha\in {\sf Comp}_n$ avoiding all compositions in ${\sf KM}$. The expansion
\[\kappa_{\alpha}=\sum_{\gamma\in {\sf Comp}_n} c_{\gamma}x^{\gamma}\] 
is \emph{multiplicity-free} if $c_{\gamma}\in \{0,1\}$ for all
$\gamma\in {\sf Comp}_n$.

\begin{theorem}
\label{thm:mfKey}
$\kappa_{\alpha}$ is multiplicity-free if and only if $\alpha \in {\overline {\sf KM}}_n$.
\end{theorem}

\begin{example}
$\alpha=(0,1,1)\in {\overline {\sf KM}}_3$. $\kappa_{\alpha}=x_2 x_3 +x_1 x_3 + x_2 x_1$
is multiplicity-free.\qed
\end{example}

\begin{example}
$\alpha=(\underline{0},2,\underline{1},\underline{2}) \not\in {\overline {\sf KM}}_4$ (contains $(0,1,2)$ in the underlined positions). 
\begin{multline}\nonumber
\kappa_{\alpha} =x_1^2 x_2^2 x_4 + x_1^2 x_2^2 x_3+2 x_1^2 x_2 x_3 x_4+x_1^2 x_2 x_4^2 +x_1^2 x_2 x_3^2 +x_1^2 x_3 x_4^2\\+ x_1^2 x_3^2 x_4 +2 x_1 x_2^2 x_3 x_4+x_1 x_2^2 x_4^2+x_1 x_2^2 x_3^2+x_1 x_2 x_3 x_4^2+x_1 x_2 x_3^2 x_4+x_2^2 x_3 x_4^2+x_2^2 x_3^2 x_4,
\end{multline}
has multiplicity.
\qed
\end{example}


Theorem~\ref{thm:mfKey} is the same as \cite[Theorem~4.10]{Hodges.Yong} (stated there without proof). In \emph{ibid.}, we initiated a study of the notion of \emph{split} multiplicity-free problems. Theorem~\ref{thm:mfKey} concerns the ``most split''
case of these problems (the ``$[n-1]$'' case, in the terminology of \emph{ibid.}). 

The sufficiency proof uses the \emph{quasi-key} model of key polynomials due to S.~Assaf--D.~Searles \cite{Assaf.Searles}.
In Section~\ref{sec:assafsearles}, we prove a preparatory theorem (Theorem~\ref{thm:qksummary}), which gives sufficient conditions
for their \emph{quasi-key polynomials} to be multiplicity-free. The conclusion of the proof of Theorem~\ref{thm:mfKey} is given in 
Section~\ref{sec:mfKey}. There, the
necessity proof uses the older \emph{Kohnert diagram} model \cite{Kohnert}. 

A.~Fink--K.~M\'esz\'aros--A.~St.~Dizier's \cite[Theorem~1.1]{FMSD} characterizes multiplicity-free \emph{Schubert polynomials} in terms of classical pattern avoidance of permutations. Since Schubert polynomials are linear combinations
of key polynomials with positive integer coefficients (see \cite[Theorem~4]{Reiner.Shimozono}), our results are related.
We do not know how to derive one result from the other. The proof methods are different.  As explained in \cite[Section~4.3]{Hodges.Yong}, one can look forward to finding ``split'' generalizations of both theorems.

\section{Quasi-key polynomials of S.~Assaf-D.~Searles} \label{sec:assafsearles}
\subsection{Multiplicity-freeness}

 \emph{Dominance order}
on $ {\sf Comp}_n$ is 
\[\alpha \geq_{\sf Dom} \beta  \text{\ \ if $\sum_{i=1}^t \alpha_i \geq \sum_{i=1}^t \beta_i$ for all
$1\leq t\leq n$}.\]

We will use notions introduced in S.~Assaf-D.~Searles' \cite{Assaf.Searles}.

\begin{definition}
A \emph{quasi-key tableau} $T$ of shape $\alpha$ fills $D(\alpha)$ with ${\mathbb Z}_{>0}$ such that

\begin{enumerate}
\item[\qkt{1}] Entries weakly decrease, left to right, along rows. Entries in row $i$ are at most $i$.
\item[\qkt{2}] Entries in each column are distinct. Entries increase upward in the first column.
\item[\qkt{3}] If $i$ appears above $k$ in the same column and $i<k$, then there is a $j$ that appears immediately to the right of that $k$, and $i < j$.
\item[\qkt{4}] If $r<s$, $\alpha_r<\alpha_s$, and $(r,c),(s,c+1)\in D(\alpha)$ then
$T(r,c)<T(s,c+1)$.
\end{enumerate}

\end{definition}

Let ${\sf qKT}(\alpha)$ be the set of quasi-key tableaux of shape $\alpha$. Given $T\in {\sf qKT}(\alpha)$, let ${\sf wt}(T)=(w_1,w_2,\ldots,w_\ell)$ where $w_i$ is the number of $i$'s appearing in $T$.

\begin{definition}
\label{def:quasiKeyPoly}
The \emph{quasi-key polynomial} $\mathfrak{D}_\alpha$ is 
\[\mathfrak{D}_\alpha = \sum_{T \in {\sf qKT}(\alpha)} x^{{\sf wt}(T)}.\]
\end{definition}

\begin{definition}
A \emph{left swap} of $\alpha\in  {\sf Comp}_n$ is $(\alpha_1,\ldots, \alpha_j,\ldots, \alpha_i,\ldots,\alpha_n)$ where
$\alpha_i < \alpha_j$ for some $i<j$. Let ${\sf lswap}(\alpha)\subseteq  {\sf Comp}_n$ be all compositions obtained by iteratively
applying (a possibly empty sequence of) left swaps to $\alpha$. For $\alpha\in {\sf Comp}_n$, let ${\sf flat}(\alpha)\in  {\sf Comp}_n$ be $\alpha$ with all  $0$'s removed. Now define 
\[{\sf Qlswap}(\alpha)=\{\gamma\in {\sf lswap}(\alpha): \text{$\gamma \leq_{\sf Dom} \tau$, for all $\tau \in {\sf lswap}(\alpha)$ such that ${\sf flat}(\gamma) = {\sf flat}(\tau)$}\}.\] 
\end{definition}

\begin{theorem}[\cite{Assaf.Searles}]\label{AssafSearlesTheorem}
\[\displaystyle \kappa_\alpha = \sum_{\beta \in {\sf Qlswap}(\alpha)} \mathfrak{D}_\beta.\]
\end{theorem}

\begin{example}
Let $\alpha=(3,2,1,3,2)$. Then
\begin{multline}\nonumber
\kappa_{\alpha}=x_1^3  x_2^2 x_3^3 x_4^2 x_5+x_1^3 x_2^2 x_3^3 x_4 x_5^2+x_1^3  x_2^3 x_3^2 x_4^2 x_5
+x_1^3 x_2^3 x_3^2  x_4 x_5^2 +x_1^3 x_2^2 x_3^2  x_4^3 x_5\\
+x_1^3 x_2^3 x_3 x_4^2 x_5^2+{\color{blue} x_1^3 x_2^2 x_3 x_4^3 x_5^2+ x_1^3 x_2^2 x_3^2 x_4^2 x_5^2}
\end{multline}
\begin{multline}\nonumber
\text{and \ } {\sf lswap}(\alpha)=\{(3,2,1,3,2),(3,3,1,2,2),(3,2,3,1,2),(3,2,2,3,1),(3,3,2,1,2),\\
(3,3,2,2,1),(3,2,3,2,1)\}.
\end{multline}
Since $\alpha$ contains no $0$'s, ${\sf Qlswap}(\alpha)={\sf lswap}(\alpha)$ (in fact, this will be the case starting
in Section~\ref{sec:6.2}). For all $\beta\in {\sf lswap}(\alpha)$, except $\beta=\alpha$, $\#{\sf qKT}(\beta)=1$; the unique
tableau is  
the \emph{super quasi-key} tableau: the one that places only $b$'s in row $b$. Hence ${\mathfrak D}_{\beta}=x^{\beta}$
in those cases. When $\beta=\alpha$ there are
two quasi-key tableaux, namely
\[\tableau{5 & 5\\4 & 4 & 4\\ 3 \\ 2 & 2\\1&1&1} \text{ \ \ and \  \ }
\tableau{5 & 5\\4 & 4 & 3\\ 3 \\ 2 & 2\\1&1&1}.
\]
Thus, ${\mathfrak D}_{\alpha}={\color{blue} x_1^3 x_2^2 x_3 x_4^3 x_5^2 + x_1^3 x_2^2 x_3^2 x_4^2 x_5^2}$. This
all agrees with Theorem~\ref{AssafSearlesTheorem}.\qed
\end{example}

 Define 
\[{\overline {\sf KM}}_n^{\geq 1}:=\{\alpha\in {\overline {\sf KM}}_n:  \alpha_i \geq 1 \text{\  for $1 \leq i \leq n$}\}.\] 

\begin{theorem}
\label{thm:qksummary}
${\mathfrak D}_{\beta}$ is multiplicity-free if $\beta\in {\sf Qlswap}(\alpha)$ and
$\alpha \in {\overline {\sf KM}}_n^{\geq 1}$.
\end{theorem}

In particular, ${\mathfrak D}_{\alpha}$ is multiplicity-free if $\alpha \in {\overline {\sf KM}}_n^{\geq 1}$.
It would be interesting to characterize precisely when ${\mathfrak D}_{\alpha}$ is multiplicity-free. D.~Brewster, H.~Raza
and the first author have conjectured that the hypothesis that $\alpha_i\geq 1$ in Theorem~\ref{thm:qksummary}
can be dropped.

The remainder of this section is devoted to the proof of Theorem~\ref{thm:qksummary}.

\subsection{Lemmas}\label{sec:6.2}
We need lemmas about ${\overline {\sf KM}}_n^{\geq 1}$, and ${\sf qKT}(\alpha)$ for 
$\alpha \in {\overline {\sf KM}}_n^{\geq 1}$. Given $\alpha\in  {\sf Comp}_n$, let 
$i_1 < \cdots < i_k$
be all indices such that $\alpha_{i_r -1} < \alpha_{i_r }$. For convenience, let  $i_0 = 1$, $i_{k+1}=n+1, \alpha_0=\infty$, and 
$\alpha_{i_{k+1}}=0$. The $m$-th \emph{segment} of $\alpha$ is 
\[ \seg{m}{}(\alpha) = \{ i_{m-1},i_{m-1}+1,\ldots,i_{m} - 1 \}, \] and it is denoted $\seg{m}{}$ when the composition is clear by context.

Define
\begin{align*}
\seg{m}{1} := & \{ b \in \seg{m}{} \, | \, \alpha_b \geq \alpha_{i_m} \},\\
\seg{m}{2} := & \{ b \in \seg{m}{}\, | \, \alpha_b < \min(\alpha_{i_{m-1} - 1}, \alpha_{i_m})\textrm{ and } b < i_{m} - 1 \},\\
 \seg{m}{3} := & \left\{
\begin{array}{cl}
\emptyset &  \text{if $m=k+1$}\\
\{ i_{m} - 1 \} & \text{otherwise.}
\end{array}\right.
\end{align*}

\begin{lemma}
\label{lemma:segPartsRange}
Let $\alpha \in  {\overline {\sf KM}}_n^{\geq 1}$. 
\begin{enumerate}
\item[(a)] $\seg{m}{} = \seg{m}{1}\sqcup \seg{m}{2}\sqcup \seg{m}{3}$.
\item[(b)] $\seg{m}{i}$ is a consecutive sequence of integers, for $i \in \{ 1,2,3 \}$.
\item[(c)] $\#\seg{m}{1} \geq 1$ for $m > 1$.
\item[(d)] If $b\in \seg{m}{3}$ (that is, $b=i_m-1$) then $\alpha_{i_{m-1}-1}\geq \alpha_b$.
\item[(e)] If $b\in \seg{m}{}$ and $\alpha_{i_{m-1}-1} < \alpha_b$, then $b\in \seg{m}{1}$.
\end{enumerate}
\end{lemma}
\begin{proof}
By definition of $\seg{m}{}$, 
$\alpha_{i_{m-1}}\geq\alpha_{i_{m-1}+1}\geq \ldots\geq
\alpha_{i_m-1}$. Thus (b) holds. For the same reason, $\seg{m}{1},\seg{m}{2},\seg{m}{3}$ are disjoint.
Since $\alpha$ avoids $(0,1,2)$, there is no $b\in \seg{m}{}$ such that
$\alpha_{i_{m-1}-1}<\alpha_b<\alpha_{i_m}$. This proves $\seg{m}{} = \seg{m}{1}\sqcup \seg{m}{2}\sqcup \seg{m}{3}$; hence
(a) holds. Next, if $m>1$ and (c) is false, then $\alpha_{i_{m-1}-1}<\alpha_{i_{m-1}}<\alpha_{i_m}$ forms
a $(0,1,2)$ pattern, a contradiction.

If (d) is false then $(\alpha_{i_{m-1}-1}<\alpha_{i_m-1}<\alpha_{i_m})$ is a $(0,1,2)$ pattern, a contradiction. Finally, (e) follows from (d), the definition of $\seg{m}{2}$, and (a).
\end{proof}

\begin{example}
Let 
$\alpha = (10,5,12,9,8,8,4,2,5,1,3)$. 
Then $\alpha\in {\overline {\sf KM}}_{11}^{\geq 1}$, $k=3$,
$i_0:=1, i_1 = 3, i_2 = 9, i_3 = 11, i_4:=12$, and
\begin{center}
$\seg{1}{} = \{ 1,2 \}$, $\seg{2}{} = \{ 3,4,5,6,7,8 \}$, $\seg{3}{} = \{ 9,10 \}$, and $\seg{4}{} = \{ 11 \}$
\end{center}
with
\begin{center}
$\begin{array}{llll}
\qquad & \seg{1}{1} = \{  \},& \seg{1}{2} = \{ 1\},& \textrm{ and }\seg{1}{3} = \{ 2 \} \\[3pt]
& \seg{2}{1} = \{ 3,4,5,6 \},& \seg{2}{2} = \{ 7 \},&\textrm{ and }\seg{2}{3} = \{ 8 \} \\[3pt]
& \seg{3}{1} = \{ 9 \},& \seg{3}{2} = \{ \},&\textrm{ and }\seg{3}{3} = \{ 10 \} \\[3pt]
& \seg{4}{1} = \{ 11 \},& \seg{4}{2} = \{ \},&\textrm{ and }\seg{4}{3} = \{  \}
\end{array}$
\end{center}\qed
\end{example}

In this proof and the sequel, it will be convenient to write, \emph{e.g.,} $(\alpha_a,\alpha_b,\alpha_c,\alpha_d)\simeq (1,0,3,2)$
if the subsequence $(\alpha_a,\alpha_b,\alpha_c,\alpha_d)$ of $\alpha$ forms a $(1,0,3,2)$ pattern.

\begin{lemma}
\label{lemma:1032avoidingConsequence}
Suppose $\alpha \in  {\overline {\sf KM}}_n^{\geq 1}$
and fix $1\leq m \leq k+1$. If $s < i_{m} - 1$ and $r > i_{m}$ then either 
\begin{itemize}
\item $\alpha_s \geq \alpha_r$; or 
\item $\alpha_s < \alpha_r$ with $\alpha_s=\alpha_{i_{m} - 1}$ and $\alpha_{i_m}=\alpha_r=\alpha_{i_{m} - 1}+1$. 
\end{itemize}
\end{lemma}
\begin{proof}
If $\alpha_s \geq \alpha_r$ we are done. Assume $\alpha_s<\alpha_r$.
We have $s < i_{m} - 1 < i_m < r$ with $\alpha_{i_{m} - 1} < \alpha_{i_{m}}$. 
Since $\alpha$ avoids $(0,1,2)$, so must the subsequence $A=(\alpha_s, \alpha_{i_{m} - 1}, \alpha_{i_{m}}, \alpha_r)$. Thus $\alpha_s \geq \alpha_{i_{m} - 1}$ and $\alpha_{i_{m}} \geq \alpha_r$. Since $A\in {\overline{\sf KM}}$, $A\not\simeq (1,0,3,2),(1,0,2,2)$. Hence $\alpha_s=\alpha_{i_{m} - 1}$. Since $A\not\simeq (0,0,2,1)$, $\alpha_{i_m}=\alpha_r$. Finally,
since $A\not\simeq (0,0,2,2)$, $\alpha_r=\alpha_{i_{m-1}}+1$.
\end{proof}

\begin{lemma}\label{lemma:rectangle}
If $D(\alpha)$ contains southwest $s\times t$ rectangle and $T\in {\sf qKT}(\alpha)$ then $T(r,c)=r$ for all $1\leq r\leq s$
and $1\leq c\leq t$.
\end{lemma}
\begin{proof}
By \qkt{1} and \qkt{2}.
\end{proof}

\begin{lemma}
\label{lemma:rowfilledrowval}
Suppose $\alpha \in {\overline {\sf KM}}_n^{\geq 1}$. Let $T \in {\sf qKT}(\alpha)$. If
\begin{itemize}
\item[(a)] $b\in \seg{1}{}$, $b\in \seg{m}{1}$ with $\alpha_{i_{m-1} - 1}\geq \alpha_b$, $b \in \seg{m}{2}$, or $b \in \seg{m}{3}$ then row $b$ of $T$ only contains $b$'s.
\item[(b)] $b\in \seg{m}{1}$ with $\alpha_{i_{m-1} - 1}<\alpha_b$ then the leftmost $\alpha_{i_{m-1} - 1} + 1$ columns of row $b$ only contain~$b$'s.
\end{itemize}
\end{lemma}
\begin{proof}
(a): First suppose $b\in \seg{1}{}$.
Row $1$ of $T$ must only contain $1$'s by \qkt{1}. If $2 \in \seg{1}{}$ then $\alpha_1 \geq \alpha_2$, so by \qkt{1} and \qkt{2} row $2$ of $T$ must only contain $2$'s. The same holds for all rows in $\seg{1}{}$, by induction. 

Now suppose we satisfy one of the other possibilities of (a). Since $b\in \seg{m}{}$,
\begin{equation}
\label{eqn:June18dee}
\alpha_r\geq \alpha_b, \ i_{m-1}\leq r\leq b.
\end{equation}
Since $\alpha$ avoids $(0,1,2)$,
\begin{equation}
\label{eqn:June18yyyy}
\alpha_r\geq \alpha_{i_{m-1}-1},  \ 1\leq r\leq i_{m-1}-1.
\end{equation}
By the hypothesis (if $b\in \seg{m}{1}$), the 
definition of $\seg{m}{2}$, or Lemma~\ref{lemma:segPartsRange}(d) (if $b\in\seg{m}{3}$),
\begin{equation}
\label{eqn:June18ywyw}
\alpha_{i_{m-1} - 1}\geq \alpha_b
\end{equation}
By (\ref{eqn:June18dee}), and by (\ref{eqn:June18yyyy}) combined with (\ref{eqn:June18ywyw}), we conclude
that $\alpha_r\geq \alpha_b$ for all $1\leq r\leq b$. Now apply Lemma~\ref{lemma:rectangle} to this $b\times \alpha_b$
southwest rectangle in $D(\alpha)$.

(b): By (\ref{eqn:June18yyyy}), the hypothesis $\alpha_{i_{m-1} - 1}<\alpha_b$, and 
(\ref{eqn:June18dee}),
\begin{equation}
\alpha_r\geq \alpha_{i_{m-1}-1}, \ 1\leq r\leq b.
\end{equation}
This implies there is a southwest $b\times \alpha_{i_{m-1}-1}$ rectangle in $D(\alpha)$. Hence by Lemma~\ref{lemma:rectangle}, $T(r,c)=r$ for $1\leq r\leq b$ and $1\leq c\leq \alpha_{i_{m-1}-1}$. Since 
$1\leq \alpha_{i_{m-1} - 1}<\alpha_s$
for $i_{m-1} - 1 < s \leq b$, we are done by \qkt{1}, \qkt{2} and \qkt{4}. 
\end{proof}

\begin{example}\label{exa:June17zzz}
Let $\alpha = (10,5,12,9,8,8,4,2,5,1,3)$. Figure~\ref{fig:June17xdd} shows the forced entries for quasi-key tableau in ${\sf qKT}(\alpha)$.\qed
\end{example}
\begin{figure}
\begin{center}
\ytableausetup{boxsize=1.4em}
\begin{ytableau}
11 & 11 & \; \\
10 \\
9 & 9 & 9 & \; & \; \\
8 & 8 \\
7 & 7 & 7 & 7 \\
6 & 6 & 6 & 6 & 6 & 6 & \; & \; \\
5 & 5 & 5 & 5 & 5 & 5 & \; & \;  \\
4 & 4 & 4 & 4 & 4 & 4 & \; & \; & \;  \\
3 & 3 & 3 & 3 & 3 & 3 & \; & \; & \; & \; & \; & \; \\
2 & 2 & 2 & 2 & 2  \\
1 & 1 & 1 & 1 & 1 & 1 & 1 & 1 & 1 & 1
\end{ytableau}
\end{center}
\caption{Forced entries of the quasi-key tableaux for Example~\ref{exa:June17zzz} \label{fig:June17xdd}}
\end{figure}

\begin{lemma}
\label{lemma:seesawLemma}
Suppose $\alpha \in {\overline {\sf KM}}_n^{\geq 1}$. Let $T \in {\sf qKT}(\alpha)$. Let ${\sf y}$ be a box such that ${\sf row}({\sf y}) > i_{m-1}$. Then $T({\sf y}) \geq i_{m-1}-1$.
\end{lemma}

\begin{proof}
To reach a contradiction, suppose 
\begin{equation}
\label{eqn:July8aop}
T({\sf y})< i_{m-1}-1.
\end{equation}
\noindent \textit{Case 1:} ($\alpha_s \geq \alpha_{{\sf row}({\sf y})}$ for $1 \leq s \leq T({\sf y})$) By Lemma~\ref{lemma:rectangle}, $T(s,c)=s$ for all $1 \leq s \leq T({\sf y})$ and $1 \leq c \leq \alpha_{{\sf row}({\sf y})}$. Since ${\sf col}({\sf y}) \leq \alpha_{{\sf row}({\sf y})}$, $T(T({\sf y}), {\sf col}({\sf y}))=T({\sf y})$. This, with \eqref{eqn:July8aop}, and the hypothesis ${\sf row}({\sf y})>i_{m-1}$, shows the label $T({\sf y})$ occurs twice in ${\sf col}({\sf y})$, contradicting $\qkt{2}$.

\noindent \textit{Case 2:} ($\alpha_s < \alpha_{{\sf row}({\sf y})}$ for some $1 \leq s \leq T({\sf y})$) 
Lemma \ref{lemma:1032avoidingConsequence} (applied to $i_{m-1}$, $r={\sf row}({\sf y})$) shows that for any $1 \leq s \leq T({\sf y})$ such that $\alpha_s < \alpha_{{\sf row}({\sf y})}$,
$\alpha_{{\sf row}({\sf y})} = \alpha_{i_{m-1}} = \alpha_{i_{m-1}-1} + 1 = \alpha_s + 1$. So,
\begin{equation}
\label{eqn:July8rft}
\alpha_s \geq \alpha_{{\sf row}({\sf y})} - 1\text{ for all $1 \leq s \leq T({\sf y})$.}
\end{equation}
Hence Lemma~\ref{lemma:rectangle} shows 
\begin{equation}
\label{eqn:July8bvc}
T(s,c)=s \text{ for all $1 \leq s \leq T({\sf y})$ and $1 \leq c \leq \alpha_{{\sf row}({\sf y})} - 1$.}
\end{equation}
Let 
\[t=\max_{s}\{1 \leq s \leq T({\sf y}), \alpha_s < \alpha_{{\sf row}({\sf y})}\};\]
$t$ is finite by this case's assumption. 
By \eqref{eqn:July8rft} (and the case assumption), $\alpha_{{\sf row}({\sf y})} - 1 = \alpha_t$. Therefore,
 by the maximality of $t$, 
\begin{equation}
\label{eqn:July8ugh}
\alpha_{{\sf row}({\sf y})} - 1 = \alpha_t < \alpha_u, \text{\  for  $t < u \leq T({\sf y})$.}
\end{equation}
Thus \eqref{eqn:July8bvc}, \eqref{eqn:July8ugh} and $\qkt{4}$ imply 
$t = T(t,\alpha_{{\sf row}({\sf y})}-1) < T(u,\alpha_{{\sf row}({\sf y})})$, for  $t < u \leq T({\sf y})$. Hence, by inductively applying $\qkt{1}$ and $\qkt{2}$ we conclude 
\begin{equation}
\label{eqn:July8fin}
T(u,\alpha_{{\sf row}({\sf y})})=u, \text{\ for  $t < u \leq T({\sf y})$.}
\end{equation}
Finally, by the definition of $t$, $\alpha_t < \alpha_{{\sf row}({\sf y})}$. So \eqref{eqn:July8bvc} and $\qkt{4}$ imply 
\begin{equation}
\label{eqn:July8qwerty}
t = T(t,\alpha_{{\sf row}({\sf y})}-1) < T({\sf row}({\sf y}),\alpha_{{\sf row}({\sf y})}).
\end{equation}
However, by \eqref{eqn:July8fin} we have $t+1=T(t+1,\alpha_{{\sf row}({\sf y})})$. Hence by (\ref{eqn:July8qwerty})
and \qkt{2}, $t+1<T({\sf row}({\sf y}),\alpha_{{\sf row}({\sf y})})$. Repeating this argument replacing $t+1$ successively with  $t+2,t+3,\ldots, T({\sf y})$
in (\ref{eqn:July8fin}) we arrive at $T({\sf y}) < T({\sf row}({\sf y}),\alpha_{{\sf row}({\sf y})})$; this
contradicts $\qkt{1}$.
\end{proof}

\begin{lemma}
\label{lemma:uniqueSmaller}
Suppose $\alpha \in {\overline {\sf KM}}_n^{\geq 1}$. Let $T \in {\sf qKT}(\alpha)$, $b \in \seg{m}{}$ for $1 < m \leq k+1$ with $\alpha_{i_{m-1}-1} < \alpha_b$, and $c > \alpha_{i_{m-1} -1}+1$. Then
\[\#\{{\sf x}\in D(\alpha):b\leq {\sf row}({\sf x}), {\sf col}({\sf x}) = c, T(x)<b\}\leq 1.\]
\end{lemma}
\begin{proof}
Suppose there were two rows 
\begin{equation}
\label{eqn:June19ppp}
b\leq r<r'  \text{\  such that $T(r',c), T(r,c)<b$.}
\end{equation}
By hypothesis, $\alpha_{i_{m-1}-1}<\alpha_b$. 
Thus, if $\alpha_b<\alpha_r$ then $(\alpha_{i_{m-1}-1},\alpha_b,\alpha_r)\simeq (0,1,2)$, contradicting
$\alpha \in {\overline {\sf KM}}_n^{\geq 1}$. Hence $\alpha_b\geq \alpha_r$. Suppose $\alpha_r<\alpha_{r'}$. Now $r \in \seg{f}{}$ for some $f \geq m$. If $\alpha_{i_{f-1} -1} \geq \alpha_r$, then Lemma~\ref{lemma:rowfilledrowval}(a) would imply row $r$ contains only $r$'s. Since this is not the case by \eqref{eqn:June19ppp}, it must be that $\alpha_{i_{f-1} -1} < \alpha_r$. Thus
$(\alpha_{i_{f-1} -1}, \alpha_r, \alpha_{r'})$ is a $(0,1,2)$ pattern, a contradiction. Therefore, 
\begin{equation}
\label{eqn:June19ggg}
\alpha_r \geq \alpha_{r'} \geq c
\end{equation} 
(where the latter inequality is by
(\ref{eqn:June19ppp})). By (\ref{eqn:June19ppp}) together with \qkt{1} and \qkt{2}, there exists 
two rows $s < s' < r$ with 
\begin{equation}
\label{eqn:June19jjj}
\alpha_s,\alpha_{s'} < c.
\end{equation} 
Since $(\alpha_s, \alpha_{s'}, \alpha_r, \alpha_{r'}) \in {\overline {\sf KM}}_4$ then it follows straightforwardly from (\ref{eqn:June19ggg}) and (\ref{eqn:June19jjj}) that 
\begin{equation}
\label{eqn:June19xyp}
\alpha_s=\alpha_{s'} \text{\ and  $\alpha_r=\alpha_{r'}=\alpha_s +1$.}
\end{equation}
In fact (\ref{eqn:June19xyp}) holds for any $s<s'<r$ satisfying (\ref{eqn:June19jjj}). In particular, by hypothesis
$\alpha_{i_{m-1}-1}<c$. Hence, there is at least one pair $s, s'$ satisfying (\ref{eqn:June19jjj}) with either $s = i_{m-1}-1$ or $s' = i_{m-1}-1$.
Then by
(\ref{eqn:June19ggg}) and 
(\ref{eqn:June19xyp}), $c\leq \alpha_r=\alpha_{i_{m-1}-1}+1$, contradicting the hypothesis on $c$.
\end{proof}

\subsection{Proof of Theorem~\ref{thm:qksummary}} The next two proposition immediately
give Theorem~\ref{thm:qksummary}.

\begin{proposition}
\label{prop:quasiKeyMultFree}
If $\alpha \in {\overline {\sf KM}}_n^{\geq 1}$, then $\mathfrak{D}_\alpha$ is multiplicity-free.
\end{proposition}

\begin{proof} Suppose not. There exists distinct $T,T'\in  {\sf qKT}(\alpha)$ such that
${\sf wt}(T)={\sf wt}(T')$. Define 
\[b:=\max\{v: \exists {\sf x}, T({\sf x})=v, T'({\sf x})\neq v\}.\]

Since ${\sf wt}(T)={\sf wt}(T')$, then
\begin{equation}
\label{eqn:June21x'}
\exists {\sf x}' \text{ \ such that \ } T'({\sf x}')=b, T({\sf x}')\neq b.
\end{equation}
Let $b'=\max\{v: \exists {\sf x}, T'({\sf x})=v, T({\sf x})\neq v\}$. We claim that $b = b'$. Firstly, \eqref{eqn:June21x'} implies $b \leq b'$. Since ${\sf wt}(T)={\sf wt}(T')$, the definition of $b'$ indicates there exists an ${\sf x}'$ such that $T({\sf x}')=b'$ with $T'({\sf x}')\neq b'$. If $b' > b$, this would contradict the definition of $b$. Hence $b=b'$ and 
\begin{equation}
\label{eqn:June20pop}
b:=\max\{v: \exists {\sf x}, T({\sf x})=v, T'({\sf x})\neq v\}=\max\{v: \exists {\sf x}, T'({\sf x})=v, T({\sf x})\neq v\}.
\end{equation}

Let
\[p_T:=\max\{c: T(b,c)=b\} \text{ \ \ and \ \ } p_{T'}=\max\{c:T'(b,c)=b\}.\]
Since $\alpha_i \geq 1$ for all $1 \leq i \leq n$, $\qkt{1}$ and $\qkt{2}$ imply that 
$T(b,1)=T'(b,1)=b$ and hence finite maximums exist 
($p_T, p_{T'}\geq 1$). By swapping $T$ and $T'$ (if necessary), we may assume 
\begin{equation}
\label{equation:wlogLess}
p_{T'} \leq p_T
\end{equation}

\begin{claim}
\label{claim:fixedAbove}
Let $b \in \seg{m}{}$ for $1 < m \leq k+1$ with $\alpha_{i_{m-1}-1} < \alpha_b$. 
\begin{itemize}
\item[(I)] $T({\sf y})\geq b$ if  ${\sf row}({\sf y})>b$ and ${\sf col}({\sf y}) \geq p_T$. Similarly,
$T'({\sf y})\geq b$ if ${\sf row}({\sf y})>b$ and ${\sf col}({\sf y}) \geq p_{T'}$.
\item[(II)] $T({\sf y})=T'({\sf y})$ if ${\sf row}({\sf y})>b$ and ${\sf col}({\sf y}) \geq p_{T}$.
\end{itemize}
\end{claim}
\noindent
\emph{Proof of Claim~\ref{claim:fixedAbove}:} (I): We prove the assertion for $T$; the $T'$ claim is the same. By definition of $p_T$ and \qkt{1}, $T(b,c)<b$ for any $c>p_T$.
By hypothesis $\alpha_{i_{m-1}-1} < \alpha_b$, and hence Lemma~\ref{lemma:segPartsRange}(e) implies $b \in \seg{m}{1}$.
Thus Lemma~\ref{lemma:rowfilledrowval}(b) indicates that 
\begin{equation}
\label{eqn:June20ggg}
p_T\geq \alpha_{i_{m-1}-1}+1.
\end{equation}
 Hence $c> \alpha_{i_{m-1}-1}+1$.
Thus, the hypotheses of Lemma \ref{lemma:uniqueSmaller} hold, and the conclusion of that lemma is that
$T({\sf y}) \geq b$ 
if ${\sf row}({\sf y})>b$ and ${\sf col}({\sf y}) =c (> p_T)$. 

Thus we may assume ${\sf row}({\sf y})>b$ and ${\sf col}({\sf y}) = p_T$. Suppose $p_T =  \alpha_b$. If $T({\sf y}) < b = T(b, p_T)$, then by \qkt{3}, there is a box of $D(\alpha)$ in position
$(b,p_T+1)=(b,\alpha_b+1)$, contradicting the definition of $D(\alpha)$. Hence, $p_T <  \alpha_b$. Let 
\begin{equation}
\label{eqn:b'def}
\ell=T(b, p_T+1) \text{\  and  \ $d = T({\sf y})$.}
\end{equation}
We want to show $d\geq b$; suppose not. By the definition (\ref{eqn:b'def}) of $\ell$ together with $\qkt{1}$, 
\begin{equation}
\label{eqn:Jul16yoa}
\ell<b.
\end{equation}
Thus there are three cases:

\noindent
\emph{Case 1:}  ($\ell \leq d < b$) $T$ violates \qkt{3} (where here $i=d, k=b$ and $j=\ell$ in that rule). 

\noindent
\emph{Case 2:} ($d<i_{m-1} - 1$) Since $b\in \seg{m}{}=\{i_{m-1},i_{m-1}+1,\ldots,i_m-1\}$ (by hypothesis), $b\geq i_{m-1}$.
Lemma~\ref{lemma:seesawLemma} states that $d=T({\sf y})\geq i_{m-1}-1$ since ${\sf row}({\sf y})>b\geq i_{m-1}$. Hence this case cannot occur.

\noindent
\emph{Case 3:} ($i_{m-1} - 1 \leq d < \ell$)  Since $b\in \seg{m}{}$ (by hypothesis), and $\ell < b$ by \eqref{eqn:Jul16yoa}, the assumption of this case says $d+1\in \seg{m}{}$.
Hence by definition of $\seg{m}{}$, \[\alpha_{d+1}\geq \alpha_{b}>\alpha_{i_{m-1}-1}\] (the latter inequality by the
hypothesis). We claim 
\begin{equation}
\label{eqn:July17tri}
T(s,p_T)=s\text{ for $d+1 \leq s \leq b$.}
\end{equation}
If $p_T = \alpha_{i_{m-1}-1}+1$, then Lemma~\ref{lemma:rowfilledrowval}(b) implies \eqref{eqn:July17tri}.  Otherwise (\ref{eqn:June20ggg}) implies $p_T > \alpha_{i_{m-1}-1}+1$. Thus 
Lemma \ref{lemma:uniqueSmaller} applied to column $p_T$ and row $d+1$ implies
\begin{equation}
\label{eqn:June20hhh}
\#\{s\geq d+1: T(s,p_T)< d+1\}\leq 1.
\end{equation}
However, $T({\sf y})=d$ and we assumed ${\sf row}({\sf y})>b>\ell\geq d+1$ (the last inequality being this case's assumption).
The previous sentence, combined with (\ref{eqn:June20hhh}) and $\qkt{1}$ says that $T(d+1,p_T)=d+1$. Iterating this
argument, using $\qkt{2}$, for $d+2 \leq s \leq b$ implies \eqref{eqn:July17tri}. 

Now apply \qkt{3} to $T({\sf y})=d<T(s,p_T)$ to see that
$T(s,p_T+1)>d$ for $d+1 \leq s \leq b$.  On the other hand, \qkt{1} says $T(s,p_T+1)\leq b$ for $d+1\leq s\leq b$. The definition of
$p_T$ means $T(s,p_T+1)\neq b$. Concluding, 
\[d<T(s,p_T+1)<b, \text{\ for $d+1\leq s\leq b$.}\] By pigeonhole, 
two of $\{T(s,p_T+1):d+1\leq s\leq b\}$ are equal, contradicting \qkt{2}.

Hence $d\geq b$, as desired.

(II): Suppose not and let $T({\sf y})\neq T'({\sf y})$ for some ${\sf y}$ such that ${\sf row}({\sf y})>b$ and
${\sf col}(y)\geq p_T$. In particular, at least one of $T({\sf y})$ and $T'({\sf y})$ is not $b$. 
If $\max\{T({\sf y}),T'({\sf y})\}<b$ we contradict (I). Hence $\max\{T({\sf y}),T'({\sf y})\}>b$. This contradicts \eqref{eqn:June20pop}. \qed

There are four possible cases to consider.

\noindent \textit{Case 1:} ($b\in \seg{1}{}$, $b\in \seg{m}{1}$ with $\alpha_{i_{m-1} - 1}\geq \alpha_b$, or $b \in \seg{m}{2}$) 
Let $b\in \seg{m}{}$ ($1\leq m\leq k+1$). By Lemma \ref{lemma:rowfilledrowval}(a), 
\begin{equation}
\label{eqn:June21aaa}
T(b,c)=b, \ T'(b,c)=b \text{\  for all $1\leq c\leq \alpha_b$.}
\end{equation}
By \qkt{1}, $b$ cannot appear in $T$ in any row $s$ strictly south of $b$. On the other hand, if $s\in \seg{m}{}$, and $s>b$, then
$\alpha_{s}\leq \alpha_b$. Hence by \qkt{2}, $b$ cannot appear in row $s$ of $T$. Now suppose $s>i_m$. Since
$i_m\in \seg{m+1}{}$, thus $i_m> b+1$. Hence by Lemma~\ref{lemma:seesawLemma}, no labels $<i_m-1$ appear in rows
strictly north of row $i_m$. In particular, $b$ does not appear in those rows. What we have just written also applies to $T'$,
thus
\begin{equation}
\label{eqn:June21jjj}
T(r,c)=b \Rightarrow r=b,i_m \text{ \ and \ } T'(r',c)\Rightarrow r'=b,i_m.
\end{equation}
Let ${\sf x}'$ be the box defined in (\ref{eqn:June21x'}). By (\ref{eqn:June21aaa}), in both $T$ and $T'$, row $b$
is filled entirely by $b$'s. Hence ${\sf row}({\sf x}')\neq b$. Thus by (\ref{eqn:June21jjj}), ${\sf row}({\sf x}')=i_m$.
Now since ${\sf wt}(T)={\sf wt}(T')$ row $i_m$ has the same number of $b$'s in $T$ and $T'$. Now, by $\qkt{1}$, all
labels left of the $b$'s in row $i_m$ of $T$ are strictly larger; the exact same statement is true of $T'$. However,
those larger labels cannot differ between $T$ and $T'$ by (\ref{eqn:June20pop}). Hence in fact, the $b$'s in row
$i_m$ are exactly in the same place in $T$ and $T'$, contradicting the definition (\ref{eqn:June21x'}) of ${\sf x}'$.

\noindent \textit{Case 2:} ($b \in \seg{m}{3}$) By Lemma \ref{lemma:rowfilledrowval}(a), 
\begin{equation}
\label{eqn:June21aaaa}
T(b,c)=b, \ T'(b,c)=b \text{\  for all $1\leq c\leq \alpha_b$.}
\end{equation}
By \qkt{1}, $b$ cannot appear in $T$ in any row $s$ strictly south of $b$. Let ${\sf x}'$ be the box defined in (\ref{eqn:June21x'}). By (\ref{eqn:June21aaaa}), in both $T$ and $T'$, row $b$
is filled entirely by $b$'s. Hence ${\sf row}({\sf x}')\neq b$. Notice 
\begin{equation}
\label{eqn:June22qqq}
T({\sf y})=T'({\sf y})  \text{ \ for all ${\sf y}$ such that ${\sf row}({\sf y}) > i_m$.}
\end{equation}
Indeed, by Lemma~\ref{lemma:seesawLemma}, $T({\sf y}), T'({\sf y})\geq i_m-1=b$. Then
(\ref{eqn:June20pop}) shows
 $T({\sf y}) = T'({\sf y})$.
 
It remains to consider if ${\sf row}({\sf x}')=i_m$ is possible. The contradiction in this case is derived exactly as
in the final four sentences of \textit{Case 1}.
 
\noindent \textit{Case 3:} ($b \in \seg{m}{1}$ for $1 < m \leq k+1$ with $\alpha_{i_{m-1}-1} < \alpha_b$, and $p_T = p_{T'}$) 
By \qkt{1} any entry in row $b$ of $T$ or $T'$ is $\leq b$. Thus, since $p_T = p_{T'}$, 
\begin{equation}
\label{eqn:June22like}
T(b,c)=b \iff 1\leq c\leq p_T \text{\ and \ } T'(b,c')=b \iff 1\leq c'\leq p_T(=p_{T'}).
\end{equation}
Hence ${\sf row}({\sf x}')\neq b$. Thus, by \qkt{1}, ${\sf row}({\sf x}') > b$. Then (\ref{eqn:June22like}) and
\qkt{2} implies ${\sf col}({\sf x}') > p_T$. Thus, Claim~\ref{claim:fixedAbove}(II) says that $T({\sf x}')=T'({\sf x}')$, which
contradicts the definition (\ref{eqn:June21x'}) of ${\sf x}'$.

\noindent \textit{Case 4:} ($b \in \seg{m}{1}$ for $1 < m \leq k+1$ with $\alpha_{i_{m-1}-1} < \alpha_b$, and $p_T > p_{T'}$) 
Since
\begin{equation}
\label{eqn:June22much}
T(b,c)=b \iff 1\leq c\leq p_T \text{\ and \ } T'(b,c')=b \iff 1\leq c'\leq p_{T'}(<p_{T}),
\end{equation}
by
\qkt{1} and ${\sf wt}(T)={\sf wt}(T')$, we have
\begin{equation}
\label{eqn:June22music}
\#\{{\sf z}\in D(\alpha): T'({\sf z})=b, T({\sf z})\neq b, {\sf row}({\sf z})>b\}\geq p_T - p_{T'}.
\end{equation}
For all ${\sf z}$ in the set from (\ref{eqn:June22music}), we have that ${\sf col}({\sf z})>p_{T'}$ by (\ref{eqn:June22much})
combined with \qkt{2}. Moreover, by Claim~\ref{claim:fixedAbove}(II), $T({\sf z})=T'({\sf z})$ if ${\sf col}({\sf z})\geq p_T$ 
and ${\sf row}({\sf z})>b$. Hence ${\sf col}({\sf z})<p_{T}$. Summarizing, by (\ref{eqn:June22music}) there are $p_T-p_{T'}$
of these boxes ${\sf z}$ that satisfy $p_{T'}<{\sf col}({\sf z})<p_T$. By pigeonhole, at least two of these ${\sf z}$
are in the same column. This contradicts \qkt{2}.

We conclude that no such $T,T'$ can exist.
\end{proof}

\begin{proposition} 
\label{prop:exhaustiveQKMultFreeCases}
Suppose $\alpha \in {\overline {\sf KM}}_n^{\geq 1}$ and $\beta \in {\sf Qlswap}(\alpha)$. Then $\mathfrak{D}_\beta$ is multiplicity-free.
\end{proposition}
\begin{proof}
 By Proposition~\ref{prop:quasiKeyMultFree}, it suffices to show that
${\sf Qlswap}(\alpha)\subseteq {\overline {\sf KM}}_n^{\geq 1}$.
Since $\alpha$ has no parts equal to $0$, ${\sf Qlswap}(\alpha) = {\sf lswap}(\alpha)$. 
Hence, by induction, it is enough to prove that if $\beta=(\ldots, \alpha_j,\ldots, \alpha_i,\ldots)$ 
is a left swap of $\alpha$ then $\beta \in {\overline {\sf KM}}_n^{\geq 1}$.  To reach a contradiction, assume $\beta \notin {\overline {\sf KM}}_n^{\geq 1}$. 

There are five cases to consider. In each Subcase, the contradiction derived is that $\alpha$ contains a pattern from
${\sf KM} = \{ (0,1,2), (0,0,2,2), (0,0,2,1), (1,0,3,2), (1,0,2,2)\}$. 

\noindent (I) \underline{$(\beta_{a_1},\beta_{a_2},\beta_{a_3})\simeq (0,1,2)$.} Since $\beta$ is a left swap of $\alpha$, 
$\{i,j\}\not\subseteq \{a_1,a_2,a_3\}$. Also, since $\alpha\in {\overline {\sf KM}}_n^{\geq 1}$,
$\{i,j\}\cap \{a_1,a_2,a_3\}\neq \emptyset$. 

\noindent \textit{Subcase 1:} ($a_1=i$) $\alpha_{i} < \alpha_{a_2} < \alpha_{a_3}$. Hence,
$(\alpha_i,\alpha_{a_2},\alpha_{a_3})\simeq (0,1,2)$.

\noindent \textit{Subcase 2:} ($a_2=i$) If $\alpha_{a_1} < \alpha_{i}$, then $\alpha_{a_1} < \alpha_{i} < \alpha_{a_3}$ and
hence $(\alpha_{a_1} ,\alpha_{i} , \alpha_{a_3})\simeq (0,1,2)$. Thus assume $\alpha_{a_1}\geq \alpha_i$. If $j > a_3$, then $(\alpha_{a_1}, \alpha_{i},\alpha_{a_3},\alpha_{j}) \simeq (0,0,2,1)$ or $\simeq (1,0,3,2)$. Otherwise $j < a_3$ and hence
 $\alpha_{i} < \alpha_{j} < \alpha_{a_3}$.

\noindent \textit{Subcase 3:} ($a_3=i$) $\alpha_{a_1} < \alpha_{a_2} < \alpha_{j}$.

\noindent \textit{Subcase 4:} ($a_1=j$) $\alpha_{i} < \alpha_{a_2} < \alpha_{a_3}$.

\noindent \textit{Subcase 5:} ($a_2=j$) If $\alpha_{a_3} > \alpha_{j}$, then $\alpha_{a_1} < \alpha_{j} < \alpha_{a_3}$. If $\alpha_{a_3} \leq \alpha_{j}$ with $i < a_1$, then $(\alpha_{i}, \alpha_{a_1}, \alpha_{j},\alpha_{a_3}) \simeq (1,0,2,2)$ or $\simeq (1,0,3,2)$. If $\alpha_{a_3} \leq \alpha_{j}$ with $i > a_1$, then $\alpha_{a_1} < \alpha_{i} < \alpha_{j}$.

\noindent \textit{Subcase 6:} ($a_3=j$)  $\alpha_{a_1} < \alpha_{a_2} < \alpha_{j}$.

\noindent (II) \underline{$(\beta_{a_1},\beta_{a_2},\beta_{a_3},\beta_{a_4})\simeq (1,0,3,2)$.}

\noindent \textit{Subcase 1:} ($a_1=i$, $a_2=j$) $\alpha_{i} < \alpha_{j} < \alpha_{a_3}$.

\noindent \textit{Subcase 2:} ($a_3=i$, $a_4=j$) $\alpha_{a_2} < \alpha_{i} < \alpha_{j}$.

\noindent \textit{Subcase 3:} ($a_1=i$ and $j \notin \{ a_1,a_2,a_3,a_4 \}$) If $\alpha_i \geq \alpha_{a_2}$, then $(\alpha_{i}, \alpha_{a_2}, \alpha_{a_3},\alpha_{a_4})$ contains either $(1,0,3,2)$ or $(0,0,2,1)$. If $\alpha_i < \alpha_{a_2}$, then $\alpha_i < \alpha_{a_2} < \alpha_{a_3}$.

\noindent \textit{Subcase 4:} ($a_2=i$ and $j \notin \{ a_1,a_2,a_3,a_4 \}$) $(\alpha_{a_1}, \alpha_{i}, \alpha_{a_3},\alpha_{a_4})\simeq (1,0,3,2)$.

\noindent \textit{Subcase 5:} ($a_3=i$ and $j \notin \{ a_1,a_2,a_3,a_4 \}$) If $\alpha_i \geq \alpha_{a_4}$, then $(\alpha_{a_1}, \alpha_{a_2}, \alpha_{i},\alpha_{a_4})\simeq (1,0,3,2)$ or $\simeq(1,0,2,2)$. If $\alpha_{a_2} < \alpha_i < \alpha_{a_4}$, then $\alpha_{a_2} < \alpha_{i} < \alpha_{j}$. If $\alpha_{a_2} \geq \alpha_i$ and $j > a_4$, then $\alpha_i < \alpha_{a_4} < \alpha_{j}$. If $\alpha_{a_2} \geq \alpha_i$ and $j < a_4$, then $(\alpha_{a_2}, \alpha_{i}, \alpha_{j},\alpha_{a_4})\simeq (1,0,3,2)$.

\noindent \textit{Subcase 6:} ($a_4=i$ and $j \notin \{ a_1,a_2,a_3,a_4 \}$) $(\alpha_{a_1}, \alpha_{a_2}, \alpha_{a_3},\alpha_{j})\simeq (1,0,3,2)$.

\noindent \textit{Subcase 7:} ($a_1=j$ and $i \notin \{ a_1,a_2,a_3,a_4 \}$) $(\alpha_{i}, \alpha_{a_2}, \alpha_{a_3},\alpha_{a_4}) \simeq (1,0,3,2)$.

\noindent \textit{Subcase 8:} ($a_2=j$ and $i \notin \{ a_1,a_2,a_3,a_4 \}$) If $\alpha_j \leq \alpha_{a_1}$, then $(\alpha_{a_1}, \alpha_{j}, \alpha_{a_3},\alpha_{a_4})\simeq (0,0,2,1)$ or $\simeq (1,0,3,2)$. If $\alpha_{a_1} < \alpha_j < \alpha_{a_4}$, then $\alpha_{i} < \alpha_{j} < \alpha_{a_4}$. If $\alpha_{a_4} \leq \alpha_j$ and $i < a_1$, then $\alpha_i < \alpha_{a_1} < \alpha_{j}$. If $\alpha_{a_4} \leq \alpha_j$ and $i > a_1$, then $(\alpha_{a_1}, \alpha_{i}, \alpha_{a_3},\alpha_{a_4}) \simeq (1,0,3,2)$.

\noindent \textit{Subcase 9:} ($a_3=j$ and $i \notin \{ a_1,a_2,a_3,a_4 \}$) $(\alpha_{a_1}, \alpha_{a_2}, \alpha_{j},\alpha_{a_4}) \simeq (1,0,3,2)$.

\noindent \textit{Subcase 10:} ($a_4=j$ and $i \notin \{ a_1,a_2,a_3,a_4 \}$) If $\alpha_j \leq \alpha_{a_3}$, then $(\alpha_{a_1}, \alpha_{a_2}, \alpha_{a_3},\alpha_{j})\simeq (1,0,3,2)$ or $\simeq (1,0,2,2)$. If $\alpha_j > \alpha_{a_3}$, then $\alpha_{a_2} < \alpha_{a_3} < \alpha_{j}$.

We leave the cases $(\beta_{a_1},\beta_{a_2},\beta_{a_3},\beta_{a_4})\simeq (1,0,2,2), (0,0,2,1), (0,0,2,2)$ 
to the reader.
\end{proof}

\section{Proof of classification theorem of multiplicity-free key polynomials}
\label{sec:mfKey}

\subsection{Kohnert diagrams and proof of necessity} \label{sec:5.1}
Assume $\alpha\not\in {\overline {\sf KM}}_n$. We will now show that $\kappa_{\alpha}$
has multiplicity.

We use \emph{Kohnert's rule} to compute $\kappa_{\alpha}$. For any $\alpha\in {\sf Comp}_n$, the \emph{skyline diagram} is
\[D(\alpha)=\{(i,j):1\leq i\leq n, 1\leq j\leq \alpha_i\}\] 
(where $i$ indexes the rows from south to north, and $j$ indexes the columns, from left to right). The set ${\sf KD}(\alpha)$ of \emph{Kohnert diagrams} is recursively defined as follows. Initially $D(\alpha)\in {\sf KD}(\alpha)$. At each stage thereafter,  given $D\in {\sf KD}(\alpha)$ a box $(i,j)\in D$ is \emph{movable} if it is
the rightmost box of $D$ in row $i$ and there exists $i'<i$ such that $(i',j)\not\in D$.
For any such movable box, a Kohnert diagram
$D'$ is obtained by replacing $(i,j)$ with $(i',j)$ where $i'$ is largest among all choices.
Generate a $D'$ from $D$ for every choice of moveable $(i,j)$. Now ${\sf KD}(\alpha)$ is the (finite) \emph{set} (not multiset) of Kohnert diagrams obtained starting from
$D(\alpha)$.

For $D\in {\sf KD}(\alpha)$ let 
\[{\sf Kohwt}(D)=\prod_{1\leq i\leq n} x_i^{\#\{j:(i,j)\in D\}}.\] 

\begin{theorem}[A.~Kohnert \cite{Kohnert}]\label{thm:KD}
$\kappa_{\alpha}=\sum_{D\in {\sf KD}(\alpha)} {\sf Kohwt}(D)$.
\end{theorem}

Given $D\in {\sf D}(\alpha)$, call a row $i$ \emph{initial} if it is empty or
the boxes in that row are precisely $(i,1),(i,2),\ldots,(i,j)$ for some $j\in {\mathbb Z}_{\geq 1}$.

\begin{lemma}
\label{lemma:earlysteps}
Suppose $D\in {\sf KD}(\alpha)$, $i'<i$ and $j'<j$. 
If $(i',j')$ and $(i,j)$ are the rightmost boxes of their rows, and all rows $i'\leq i''<i$ are initial, then $D'$ obtained by replacing $(i,j)$ with $(i',j)$, is in ${\sf KD}(\alpha)$.
\end{lemma}
\begin{proof}
The hypotheses on $i,j,i',j'$ guarantee that $(i,j)$ is moveable. Since each row
$i''$ is initial, $(i,j)\to (i'',j)$ is a Kohnert move (giving a diagram $D''$) whenever
$j''<j$ (where $(i'',j'')$ is the rightmost box of row $i''$). If indeed $i''=i'$ we are done, \emph{i.e.},
$D=D''$. Otherwise since $i''$ is initial in $D(\alpha)$, it must be that $(i'',j)$ is the rightmost box of its row in $D''$, and is therefore moveable. Then since all 
rows $i'<i'''<i''$ are initial in $D''$ (since they were in $D$), the claim follows by
induction on $i-i'$.
\end{proof}

\begin{corollary}
\label{cor:earlysteps}
Suppose $D=D(\alpha)$, $i'<i$, $j'<j$. Let $(i',j')$ and $(i,j)$ be the rightmost boxes
of their rows. Then $D'$ as defined in Lemma~\ref{lemma:earlysteps} is in ${\sf KD}(\alpha)$.
\end{corollary}
\begin{proof}
All rows of $D(\alpha)$ are initial, so Lemma~\ref{lemma:earlysteps} applies.
\end{proof}

In what follows, $(i_k,j_k)$ is the rightmost box of row $i_k$.

\noindent \textit{Case 1:}  ($\alpha$ contains $(0,1,2)$.) Let $i_0<i_1<i_2$ be the rows of the ``$0$'',
``$1$'' and ``$2$'' respectively. By Corollary~\ref{cor:earlysteps} we can  
replace $(i_2,j_2)$ with $(i_0,j_2)$ in $D(\alpha)$, resulting in $D'\in {\sf Koh}(\alpha)$.
On the other hand, by Corollary~\ref{cor:earlysteps} one can obtain $E'\in {\sf Koh}(\alpha)$ by replacing $(i_1,j_1)$ with $(i_0,j_1)$. Since the rows $r\geq i_1$ in $E'$ are still
initial, we can apply Lemma~\ref{lemma:earlysteps} to obtain $E''\in {\sf KD}(\alpha)$
by replacing $(i_2,j_2)$ with $(i_1,j_2)$ in $D''$. The net effect in both cases is to place an additional box in row $i_0$ and remove a box from row $i_2$. Hence ${\sf Kohwt}(D')={\sf Kohwt}(E'')$; therefore by Theorem~\ref{thm:KD}, 
 $[{\sf Kohwt}(D')] \kappa_{\alpha}\geq 2$, as desired.

\noindent \textit{Case 2:}  ($\alpha$ contains $(0,0,2,1)$.) Let $i_{0'}<i_0<i_2<i_1$ be the indices  of the $(0,0,2,1)$ pattern (in the respective order). By Corollary~\ref{cor:earlysteps}, $D'$ obtained from $D(\alpha)$ by the swap $(i_1,j_1)\to (i_{0'},j_1)$ is in ${\sf KD}(\alpha)$. Since
 all rows $r\geq i_0$ of $D''$ are initial, we can use Lemma~\ref{lemma:earlysteps} to 
 move $(i_2,j_2)\to (i_0,j_2)$ giving $D'''\in {\sf KD}(\alpha)$. On the other hand,
 starting from $D(\alpha)$ we can again use Corollary~\ref{cor:earlysteps} to define
 $E'\in {\sf KD}(\alpha)$ by the swap $(i_2,j_2)\to (i_{0'},j_2)$. Then Lemma~\ref{lemma:earlysteps} allows us to move $(i_1,j_1)\to (i_0,j_1)$ giving $E''\in {\sf KD}(\alpha)$. Now one can see $D'''\neq E''$ but both have the same ${\sf Kohwt}$, as
 needed.
 
 \noindent \textit{Case 3:}  ($\alpha$ contains $(1,0,3,2)$.) This is the same argument as 
 {\sf Case 2} except that we use $i_1<i_0<i_3<i_2$ in place of $i_{0'}<i_0<i_2<i_1$,
 respectively.
 
 \noindent \textit{Case 4:}  ($\alpha$ contains $(0,0,2,2)$.) Let $i_{0'}<i_0<i_{2'}<i_2$ be the
 rows of $(0,0,2,2)$ in that respective order. By Corollary~\ref{cor:earlysteps} turn
 $D(\alpha)$ into $D'\in {\sf KD}(\alpha)$ by the move $(i_{2'},j_{2'})\to (i_{0'},j_{2'})$.
 Now since all rows $r\geq i_0$ of $D'$ are initial, we can apply the argument of
 {\sf Case 1} using rows $i_0<i_{2'}<i_2$ rather than the $i_0<i_1<i_2$ of that case.
 
 \noindent \textit{Case 5:}  ($\alpha$ contains $(1,0,2,2)$.) Let $i_{0'}<i_0<i_{2'}<i_2$ be the
 rows of $(0,0,2,2)$ in that respective order. We can apply the argument of
 {\sf Case 4}, where these indices play the role of $i_{0'}<i_0<i_{2'}<i_2$
 from that case.

This completes the necessity argument. 

\subsection{Proof of sufficiency}

Let $b \in \seg{m}{}$ for some $1\leq m \leq k+1$. Define
\begin{equation}
\label{equation:rmin}
\rmin{b}{\alpha}=\left\{\begin{array}{ll}
        b & \text{$\alpha_{i_{m-1}-1} \geq \alpha_b$} \\[2pt]
        i_{m-1}-1 \qquad & \text{otherwise} \\
                \end{array} \right.
\end{equation}
\begin{equation}
\label{equation:rmax}
\rmax{b}{\alpha}=\left\{\begin{array}{ll}
        b+1 \qquad & \text{if } \alpha_{b+1} \geq \alpha_{i_m} \\[2pt]
        i_m \qquad & \text{otherwise} \\
        \end{array} \right.
\end{equation}
\begin{equation}
\label{equation:flex}
\flex{b}{\alpha}=\left\{\begin{array}{ll}
        \alpha_{\rmax{b}{\alpha}} - \alpha_{\rmin{b}{\alpha}} - 1 \qquad & \text{if } \alpha_{\rmax{b}{\alpha}} > \alpha_{\rmin{b}{\alpha}} \\[2pt]
        0 \qquad & \text{otherwise}
        \end{array} \right.
\end{equation}

\begin{example}\label{exa:rminmaxflex}
Consider $\alpha = (10,5,12,9,8,8,4,2,5,1,3)$. Then $k=3$,
$i_0:=1, i_1 = 3, i_2 = 9, i_3 = 11, i_4:=12$.
Let $b=5 \in \seg{2}{}$. Then $\rmin{b}{\alpha} = i_1 = 2$, since $\alpha_{i_1} < \alpha_b$. Since $\alpha_{b+1} > \alpha_{i_2}$, $\rmax{b}{\alpha} = b+1 = 6$. Thus $\flex{b}{\alpha} = \alpha_{\rmax{b}{\alpha}} - \alpha_{\rmin{b}{\alpha}} - 1 = \alpha_6 - \alpha_2 - 1 = 2$.
\end{example}

\begin{lemma}
\label{lemma:dominanceInterval}
Suppose $\alpha \in  {\overline {\sf KM}}_n^{\geq 1}$ and $T \in {\sf qKT}(\alpha)$ with $\mu = {\sf wt}(T)$. Then 
\begin{itemize}
\item[(I)] $\alpha \leq_{\sf Dom} \mu$. 
\item[(II)] $\mu_1+\cdots+\mu_b \leq \alpha_1+\cdots+\alpha_b+\flex{b}{\alpha}$, for $1 \leq b \leq n$.
\end{itemize}
\end{lemma}
\begin{proof}
(I): By $\qkt{1}$, $\alpha_1+\cdots+\alpha_b \leq \mu_1+\cdots+\mu_b$, for $1\leq b\leq n$. Thus $\alpha \leq_{\sf Dom} \mu$.

(II):
Any $T \in {\sf qKT}(\alpha)$ only uses entries $\leq b$ in the first $b$ rows. Thus
\begin{equation}
\label{eqn:June22ggg}
\#\{{\sf x}:T({\sf x})\leq b, {\sf row}({\sf x})\leq b\}=\alpha_1+\ldots +\alpha_b.
\end{equation}
Define 
\begin{equation}
\label{eqn:flexset}
F_{\alpha}(b) = \{{\sf x}:T({\sf x})\leq b, {\sf row}({\sf x}) > b\}.
\end{equation}
We claim that 
\begin{equation}
\label{eqn:June22flexmeaning}
\#F_{\alpha}(b)\leq \flex{b}{\alpha}.
\end{equation}
Clearly (\ref{eqn:June22ggg}) combined with (\ref{eqn:June22flexmeaning}) proves (II).

It remains to prove (\ref{eqn:June22flexmeaning}). Fix $b$; thus
$b \in \seg{m}{}$ for some $1\leq m \leq k+1$. Let ${\sf y}\in F_{\alpha}(b)$.
If $b \notin \seg{m}{3}$ then $b<i_{m}-1$. Lemma \ref{lemma:seesawLemma} asserts that $b$ cannot appear in rows
strictly greater than $i_m$, \emph{i.e.}, $i_m \geq {\sf row}({\sf y})$. Therefore, in view of the definition (\ref{eqn:flexset}), we have
\[b<{\sf row}({\sf y})\leq i_m.\] 
By definition of $\seg{m}{}$, $\max_{b<r\leq i_m}\{\alpha_r\}=\alpha_{{\sf rmax}_{\alpha}(b)}$. Thus 
\begin{equation}
\label{eqn:June22vvv}
{\sf col}({\sf y})\leq \alpha_{{\sf rmax}_{\alpha}(b)}.
\end{equation}
If $b \in \seg{m}{3}$, then since $\alpha$ is $(0,1,2)$ avoiding, 
$\max_{b<r\leq n}\{\alpha_r\}=\alpha_{i_m}=\alpha_{{\sf rmax}_{\alpha}(b)}$. Thus (\ref{eqn:June22vvv}) holds for all $b$.

We claim 
\begin{equation}
\label{eqn:June22theclaim}
{\sf col}({\sf y}) \geq \alpha_{\rmin{b}{\alpha}}+2.
\end{equation}

\noindent \textit{Case 1:} ($\alpha_{i_{m-1}-1} \geq \alpha_b$) 
By this case's assumption, and the definition (\ref{equation:rmin}), $\rmin{b}{\alpha}=b$. If $i_{m-1}\leq s\leq b$
then by definition of $\seg{m}{}$,
\begin{equation}
\label{eqn:June22qwe}
\alpha_s\geq \alpha_b=\alpha_{\rmin{b}{\alpha}}.
\end{equation}
If $s= i_{m-1}-1$, then (\ref{eqn:June22qwe}) holds by the assumed inequality of this case. Finally, if 
$s<i_{m-1}-1$ then $\alpha_s\geq \alpha_{i_{m-1}-1}$ since $\alpha$ avoids $(0,1,2)$; hence (\ref{eqn:June22qwe})
holds again. By Lemma~\ref{lemma:rectangle}, $T(r,c)=r$ for $1\leq r\leq b$ and $1\leq c\leq \alpha_b$.
Thus by \qkt{2}, $T(r,c)>b$ whenever $r>b$ and $c\leq \alpha_b$.  Lastly, if there exists $r>b$ such that $\alpha_r > \alpha_{\rmin{b}{\alpha}}$, then \qkt{4} implies $T(r,\alpha_{\rmin{b}{\alpha}} +1)>b$. Therefore (\ref{eqn:June22theclaim}) holds.

\noindent \textit{Case 2:} ($\alpha_{i_{m-1}-1} < \alpha_b$) 
By this case's assumption, and the definition (\ref{equation:rmin}), $\rmin{b}{\alpha}=i_{m-1}-1$.
If $i_{m-1}\leq s<b$ then $\alpha_s\geq \alpha_{b}\geq \alpha_{i_{m-1}-1}=\alpha_{\rmin{b}{\alpha}}$, and in particular:
\begin{equation}
\label{eqn:June22cdef}
\alpha_s\geq \alpha_{i_{m-1}-1}=\alpha_{\rmin{b}{\alpha}}.
\end{equation}
If $s=i_{m-1}-1$ then (\ref{eqn:June22cdef}) holds trivially. Finally if $1\leq s<i_{m-1}-1$ then since
$\alpha$ is $(0,1,2)$-avoiding, $\alpha_s\geq \alpha_{i_{m-1}-1}$ and (\ref{eqn:June22cdef}) again holds.
Hence by Lemma~\ref{lemma:rectangle}, $T(r,c)=r$ for $1\leq r\leq b$ and $1\leq c\leq \alpha_{i_{m-1}-1}$.
Thus by \qkt{2}, $T(r,c)>b$ whenever $r>b$ and $c\leq \alpha_{i_{m-1}-1}$. If there exists $r>b$ such that $\alpha_r > \alpha_{\rmin{b}{\alpha}}=\alpha_{i_{m-1}-1}$, then \qkt{4} implies $T(r,\alpha_{\rmin{b}{\alpha}} +1)>b$. Therefore (\ref{eqn:June22theclaim}) holds.

By (\ref{eqn:June22vvv}) and (\ref{eqn:June22theclaim}),
\begin{equation}
\label{eqn:June23aaa}
\alpha_{\rmin{b}{\alpha}}+2 \leq {\sf col}({\sf y}) \leq \alpha_{\rmax{b}{\alpha}}.
\end{equation}

Therefore ${\sf y}$ can appear in $\leq \flex{b}{\alpha}$ columns of $T$. If we show at most one ${\sf y} \in F_{\alpha}(b)$ can have ${\sf col}({\sf y})=c$, for $\alpha_{\rmin{b}{\alpha}}+2 \leq  c \leq \alpha_{\rmax{b}{\alpha}}$ then (\ref{eqn:June22flexmeaning}) follows.

\noindent \textit{Case 1:} ($\alpha_{i_{m-1}-1} \geq \alpha_b$) 
By this case's assumption, and definition (\ref{equation:rmin}), $\rmin{b}{\alpha}=b$. If $\alpha_{\rmin{b}{\alpha}} \geq \alpha_{\rmax{b}{\alpha}}$, then by (\ref{equation:flex}) $\flex{b}{\alpha}=0$. By (\ref{eqn:June23aaa}) $\#F_{\alpha}(b)=0$. And hence, $\#F_{\alpha}(b) \leq \flex{b}{\alpha}$, proving (\ref{eqn:June22flexmeaning}).

Thus we assume $\alpha_{\rmin{b}{\alpha}} < \alpha_{\rmax{b}{\alpha}}$. If $b \in \seg{m}{3}$, then $\rmax{b}{\alpha}=i_m=b+1$. Otherwise, $\alpha_{b+1} \leq \alpha_b = \alpha_{\rmin{b}{\alpha}} < \alpha_{\rmax{b}{\alpha}}$, where the first inequality follows from the definition of $\seg{m}{}$. This implies $\rmax{b}{\alpha} \neq b+1$. Thus, by (\ref{equation:rmax}), $\rmax{b}{\alpha}=i_m$. 

Let $b < s < i_m$. Then $\alpha_{i_{m-1}-1} \geq \alpha_b \geq \alpha_s$. So Lemma~\ref{lemma:rowfilledrowval}(a), applied to $s$ says that only $s$'s appear in row $s$; in particular
there are no entries $\leq b$ in row $s$. Thus if ${\sf y} \in F_{\alpha}(b)$, then 
\begin{equation}
\label{eqn:June23abc}
{\sf row}({\sf y}) \geq i_m = \rmax{b}{\alpha}.
\end{equation}

We apply Lemma~\ref{lemma:uniqueSmaller} to row $i_m \in \seg{m+1}{}$, and $c \geq \alpha_{\rmin{b}{\alpha}}+2 = \alpha_{b}+2 > \alpha_{i_{m}-1} + 1$ (the final inequality follows from the definition of $\seg{m}{}$). We conclude that at most one ${\sf y} \in F_{\alpha}(b)$ can have ${\sf row}({\sf y}) \geq i_m$ and ${\sf col}({\sf y})=c$. 
This, combined with (\ref{eqn:June23abc}), implies (\ref{eqn:June22flexmeaning}).

\noindent \textit{Case 2:} ($\alpha_{i_{m-1}-1} < \alpha_b$) 
By this case's assumption, and (\ref{equation:rmin}), $\rmin{b}{\alpha}=i_{m-1}-1$. If $\alpha_{\rmin{b}{\alpha}} \geq \alpha_{\rmax{b}{\alpha}}$, we get $\#F_{\alpha}(b)=0=\flex{b}{\alpha}$ in exactly the same way as the previous case. 

Hence we again assume $\alpha_{\rmin{b}{\alpha}} < \alpha_{\rmax{b}{\alpha}}$. If $\rmax{b}{\alpha}=b+1$, then by our assumption $\alpha_{i_{m-1}-1} = \alpha_{\rmin{b}{\alpha}} < \alpha_{\rmax{b}{\alpha}} = \alpha_{b+1}$. Thus we may apply Lemma~\ref{lemma:uniqueSmaller} to row $b+1$, and $c \geq \alpha_{\rmin{b}{\alpha}}+2 > \alpha_{i_{m-1}-1} + 1$. We conclude that at most one ${\sf y} \in F_{\alpha}(b)$ can have ${\sf col}({\sf x})=c$. 

Otherwise, $\rmax{b}{\alpha}=i_m$. 

\noindent
\emph{Subcase 2.1} ($\alpha_{i_{m-1}-1} \geq \alpha_{b+1}$) Let $b < s < i_m$. Then $\alpha_{i_{m-1}-1} \geq \alpha_{b+1} \geq \alpha_s$. Thus Lemma~\ref{lemma:rowfilledrowval}(a), applied to $s$ 
says that only $s$'s appear in row $s$; 
there are no entries $\leq b$ in row $s$. Thus if ${\sf y} \in F_{\alpha}(b)$, then 
\begin{equation}
\label{eqn:June23abc2}
{\sf row}({\sf y}) \geq i_m = \rmax{b}{\alpha}.
\end{equation}

We have $\alpha_{i_{m-1}-1} = \alpha_{\rmin{b}{\alpha}} < \alpha_{\rmax{b}{\alpha}}=\alpha_{i_m}$. We apply Lemma~\ref{lemma:uniqueSmaller} to row $i_m \in \seg{m+1}{}$, and $c \geq \alpha_{\rmin{b}{\alpha}}+2 > \alpha_{i_{m-1}-1} + 1 \geq \alpha_{i_{m}-1} + 1$ (the final inequality holds since $\alpha$ avoids $(0,1,2)$). The lemma implies  that at most one ${\sf y} \in F_{\alpha}(b)$ can have ${\sf row}({\sf y}) \geq i_m$ and ${\sf col}({\sf x})=c$. This, combined with (\ref{eqn:June23abc2}), implies (\ref{eqn:June22flexmeaning}).

\noindent
\emph{Subcase 2.2} ($\alpha_{i_{m-1}-1} < \alpha_{b+1}$) Apply Lemma~\ref{lemma:uniqueSmaller} to row $b+1$, and $c \geq \alpha_{\rmin{b}{\alpha}}+2 > \alpha_{i_{m-1}-1} + 1$. So, at most one ${\sf y} \in F_{\alpha}(b)$ has ${\sf row}({\sf y}) \geq b + 1$ and ${\sf col}({\sf x})=c$, proving (\ref{eqn:June22flexmeaning}). 
\end{proof}

Given $\gamma \in {\sf lswap}(\alpha)$, define a \emph{left swap sequence} of $\gamma$ to be $\gamma^{(0)},\ldots,\gamma^{(t)}$ such that $\gamma^{(0)} = \alpha$, $\gamma^{(t)} = \gamma$, and $\gamma^{(i+1)}$ is a left swap of $\gamma^{(i)}$ for $0 \leq i < t$.

\begin{lemma}
\label{lemma:lswapSeq}
Let $\gamma,\tau \in {\sf lswap}(\alpha)$ with $\gamma >_{lex} \tau$ and $b=\min \{ i : \gamma_i > \tau_i \}$. 
\begin{itemize}
\item[(I)] There exists a left swap sequence of $\gamma$ equal to $\gamma^{(0)},\ldots,\gamma^{(t)}$ and of $\tau$ equal to 
$\tau^{(0)},\ldots,\tau^{(u)}$ such that $\beta = \gamma^{(i)} = \tau^{(j)}$ with $\beta_s = \gamma_s = \tau_s$ for all $1 \leq s < b$.
\item[(II)] $\gamma,\tau\in {\sf lswap}(\beta)$.
\item[(III)] No left swap sequence from $\beta$ to $\gamma$ (or $\tau$) involves the indices $1,2,\ldots,b-1$.
\end{itemize}
\end{lemma}
\begin{proof}
(I):  Given $(x_1,\ldots,x_k)\in {\mathbb Z}_{\geq 0}^k$ let $(x_1',\ldots,x_k')$ be the coordinates sorted into 
weakly increasing order ($x_1'\leq x_2'\leq\ldots \leq x_k'$). Given $(x_1,\ldots,x_k), (y_1,\ldots,y_k)\in {\mathbb Z}_{\geq 0}^k$ we write 
\[(x_1,\ldots,x_k)\preceq (y_1,\ldots,y_k) \text{\  if $x_i'\leq y_i'$ for $1\leq i\leq k$.}\]

Recall (see, e.g., \cite[Proposition~2.1.11]{Manivel}), if $u,v\in {\mathfrak S}_n$ then 
\begin{equation}
\label{eqn:June24bruhat}
u\leq v  \iff (u(1),u(2),\ldots,u(k))\preceq
(v(1),v(2),\ldots,v(k)) \text{\  for $1\leq k\leq n$.}
\end{equation}
In particular, if 
\begin{equation}
\label{eqn:June24people}
u\leq v \Rightarrow u(1)\leq v(1)
\end{equation}

\noindent
\emph{Special Case:} ($\alpha\in {\mathfrak S}_n$) Hence $\alpha\leq \gamma,\tau$ (Bruhat order). Induct on $n$, the base $n=1$ being trivial. Suppose $n>1$. If $b=0$ (\emph{i.e.}, $\gamma(1)\neq \tau(1)$), $\beta=\alpha$ satisfies
(I). Thus assume $b\geq 1$, and let $T=\gamma(1)=\tau(1)$. Thus, by (\ref{eqn:June24people}), 
$\alpha(1)\leq T$. Let $\alpha(j)=T$. There exists a sequence of left swaps that show
\[\alpha\leq \alpha':=T \ \alpha(2) \ \alpha(3) \ldots \alpha(j-1)\  \alpha(j+1) \ldots \alpha(n).\]
It is straightforward from (\ref{eqn:June24bruhat}) that $\alpha'\leq \gamma,\tau$. Let
$\overline{\alpha'}, \overline{\gamma}, \overline{\tau}$ be the list of rightmost $n-1$ entries of $\alpha',\gamma,\tau$ (respectively); these are permutations of ${\mathfrak S}_{n-1}$ on $[n]-\{T\}$. By induction, obtain $\overline{\beta}\in {\mathfrak S}_{n-1}$
satisfying (I) with respect to $\overline{\alpha'}, \overline{\gamma}, \overline{\tau}$. Then 
\[\beta:=T \overline{\beta}(1) \overline{\beta}(2)\ldots \overline{\beta}(n-1)\in {\sf lswap}(\alpha)\]
satsfies (I) with respect to $\alpha,\beta,\gamma$, as desired (the two swap sequences to go from $\alpha\to \gamma$
and $\alpha\to \tau$ being the ``concatentation'' of the sequence from $\alpha\to \alpha'$ with the two sequences
that send $\overline{\alpha'}\to \overline{\gamma}$ and $\overline{\alpha'} \to \overline\tau$, interpreted in the obvious manner).

\noindent
\emph{General Case:} ($\alpha\in {\sf Comp}_n$)  Pick any 
$\widehat{\alpha}\in {\mathfrak S}_n$ with the property that 
\begin{equation}
\label{eqn:June24swap}
\widehat{\alpha}(i)<\widehat{\alpha}(j)\iff \alpha(i)\leq \alpha(j).
\end{equation}

Define $\widetilde\gamma\in {\mathfrak S}_n$ by applying the same  sequence (in terms of positions) of left swaps to $\widetilde{\alpha}$ used to
obtain $\gamma$ from $\alpha$. Similarly, define $\widetilde\tau$. By the definition of left swaps and (\ref{eqn:June24swap}), 
\[\widehat{\alpha}\leq \widetilde\gamma,\widetilde\tau
\]
(if $u,v\in {\mathfrak S}_n$ and $v$ is obtained from $u$ by a left swap, then $u\leq v$; this follows from, \emph{e.g.}, (\ref{eqn:June24bruhat})). Call two labels $i,j\in [n]$ to be $\alpha$-equivalent ($i\equiv j$) if $\alpha(i)=\alpha(j)$. Now apply left swaps to $\widehat\gamma$ so that all equivalent labels  are
in decreasing order. Similarly one defines $\widehat\tau$.

\begin{example} If $\alpha=(2,2,4,2,2,4)$ then $\widehat{\alpha}=125346$. Consider $\alpha=(2,2,4,{\color{blue} 2},2,{\color{blue} 4})\to ({\color{blue} 2},2,4,{\color{blue} 4},2,2)\to (4,2,4,2,2,2)=\gamma$. 
Then $\widetilde\gamma=625143$. Here $\{1,2,3,4\}$ and $\{5,6\}$ are the two $\alpha$-equivalence classes.
So $\widehat\gamma=645321$.\qed
\end{example}

By simple considerations about Bruhat order, 
\begin{equation}
\label{eqn:July11ggg}
\widehat{\alpha}\leq \widetilde\gamma\leq \widehat\gamma, \text{ \ and 
$\widehat{\alpha}\leq \widetilde\tau\leq \widehat\tau$}.
\end{equation}

By construction, 
\begin{equation}
\label{eqn:June24sss}
\widehat\gamma>_{lex} \widehat\tau
\end{equation}
and
\begin{equation}
\label{eqn:June24ttt}
\widehat\gamma(i)=\widehat\tau(i), \text{ \ for $1\leq i \leq b$.}
\end{equation}
In view of (\ref{eqn:July11ggg}), (\ref{eqn:June24sss}) and (\ref{eqn:June24ttt}) we can apply the Special Case
to construct $\widehat\beta\in {\mathfrak S}_n$, and left swap sequences, that satisfy (I) with respect to $\widehat\alpha, \widehat\gamma,\widehat\tau$.

Define $\beta$ from $\widehat\beta$ by replacing the label $i$ with $\alpha(i)$. To define the swap sequence from
$\alpha$ to $\beta$ we apply the same left swaps (\emph{i.e.}, interchange the same positions) as the sequence
from $\widehat{\alpha}$  to $\widehat{\beta}$, with the exception that we skip left swaps of the underlying permutations that involve equivalent labels. 
Similarly, one defines continuation of this sequence to $\widehat \gamma$, and separately, to $\widehat\tau$. The claim
follows.

(II): This is trivial from (I).

(III): Suppose such a left swap exists, say with $i\in [1,b-1]$. We may assume $i$ is the minimal such
index. Then $\tau_i=\gamma_i=\beta_i$ is replaced with a strictly larger number, and it follows that
$\gamma_i>\beta_i$, a contradiction.
\end{proof}

\begin{lemma}
\label{lemma:minLswap}
Assume $\alpha\in {\overline {\sf KM}}_n^{\geq 1}$.
Let $\gamma,\tau \in {\sf lswap}(\alpha)$ with $\gamma >_{lex} \tau$ and $b=\min \{ i : \gamma_i > \tau_i \}$. Let $b \in \seg{m}{}(\beta)$ and $\beta$ be from Lemma~\ref{lemma:lswapSeq}. There exists a (minimum) index $r > b$ such that
\begin{equation}
\label{eqn:June25zzz}
\gamma_b = \beta_r > \beta_b.
\end{equation}

Let $i_1(\beta)<i_2(\beta)<\ldots<i_{m}(\beta)<\ldots$ be such that $\alpha_{i_r(\beta)-1}<\alpha_{i_{r}(\beta)}$. The indices $b,r$ satisfy one of the following.
\begin{itemize}
\item[(A)] $b \in \seg{m}{2}(\beta)$ and $r \geq i_{m}(\beta)$,
\item[(B)] $b\in \seg{m}{3}(\beta)$ and $r \in \seg{m+1}{1}(\beta)$,
\item[(C)] $b\in \seg{m}{3}(\beta)$ and $r \geq i_{m+1}(\beta)$.
\end{itemize}
\end{lemma}
\begin{proof}
By Lemma~\ref{lemma:lswapSeq}, $\gamma, \tau \in {\sf lswap}(\beta)$. Hence, by the reasoning in the proof of
that lemma (the characterization of Bruhat order) and the definition of $b$,
\begin{equation}
\label{eqn:June25fgh}
\gamma_b, \tau_b \geq \beta_b.
\end{equation}

If $\gamma_b = \beta_b$, then $\tau_b \geq \gamma_b$ (by \eqref{eqn:June25fgh}). This contradicts the definition of $b$, hence $\gamma_b > \beta_b$. This, combined with the definition of left swaps, implies \eqref{eqn:June25zzz}.

\noindent
\emph{Case 1:} ($b \in \seg{m}{1}(\beta)$) By the definition of $\seg{m}{1}(\beta)$ (and a simple induction), $\beta_b \geq \beta_s$ for all $b<s$. This contradicts \eqref{eqn:June25zzz}. That is, this case cannot actually occur.

\noindent
\emph{Case 2:} ($b \in \seg{m}{2}(\beta)$) By definition of $\seg{m}{}$, $r\not\in \seg{m}{}$. Hence 
$r\geq i_m(\beta)\in \seg{m+1}{}$; this is (A).

\noindent
\emph{Case 3:} ($b \in \seg{m}{3}(\beta)$)
By definition, $r\not\in \seg{m+1}{2}(\beta)$. If $r\in \seg{m+1}{3}(\beta)$
then
\[(\beta_{i_{m}(\beta)-1}, \beta_{i_{m+1}(\beta)-1},\beta_{i_{m+1}(\beta)})\simeq (0,1,2),\] 
 a contradiction.
Thus, only (B) and (C) as possible, as desired.
 \end{proof}

\begin{lemma}
\label{lemma:seg1big}
If $\alpha\in {\overline {\sf KM}}_n^{\geq 1}$, $i\in \seg{m}{1}$ ($1<m\leq k+1$), and
$\alpha_{i_{m-1}-1}<\alpha_i$ then $\alpha_i\geq \alpha_j$
for all $i\leq j$.
\end{lemma}
\begin{proof}
Say $i<j$ but $\alpha_i<\alpha_j$. Then $(\alpha_{i_{m-1}-1},\alpha_i,\alpha_j)\simeq (0,1,2)$,
contradicting $\alpha\in {\overline {\sf KM}}_n^{\geq 1}$.
\end{proof}

\begin{proposition}
\label{theorem:dominanceIntervalDisjoint}
Let $\alpha\in {\overline {\sf KM}}_n^{\geq 1}$.
If $\gamma,\tau \in {\sf lswap}(\alpha)$ and $\gamma >_{lex} \tau$ there exists $z\in [1,n]$ such that 
\begin{equation}
\label{eqn:June29thegoal}
\tau_1+\cdots+\tau_z+\flex{z}{\tau} < \gamma_1+\cdots+\gamma_z.
\end{equation}
\end{proposition}
\begin{proof}
Our analysis is based on the cases (A), (B), (C) from Lemma \ref{lemma:minLswap}, as well as the notation from that lemma.
By the (proof of) Proposition~\ref{prop:exhaustiveQKMultFreeCases}, 
\begin{equation}
\label{eqn:June29yyy}
\beta\in {\overline {\sf KM}}_n^{\geq 1}
\end{equation}

\noindent 
\emph{Case (A):} ($b \in \seg{m}{2}(\beta)$ and $r \geq i_m(\beta)$) Let $t > b$ such that $\beta_t > \beta_b$. By the definition of $\seg{m}{}$, $t \geq i_m(\beta)$. If $t > i_m(\beta)$, then by (\ref{eqn:June29yyy}) we can
apply Lemma \ref{lemma:1032avoidingConsequence} to $\beta$, $m$, $r = t$, and $s = b$. The 
lemma concludes that $\beta_{i_m(\beta)}=\beta_t$. Otherwise, $t=i_m(\beta)$, and $\beta_t = \beta_{i_m(\beta)}$. Thus,
\begin{equation}
\label{eqn:June29bbb}
\beta_t = \beta_{i_m(\beta)} \text{ \ for all $t>b$ such that $\beta_t>\beta_b$.}
\end{equation}
By Lemma~\ref{lemma:seg1big},
\begin{equation}
\label{eqn:June29ccc}
\tau_b < \gamma_b = \beta_{r} \leq \beta_{i_m(\beta)}.
\end{equation}
Consider a sequence of left swaps transforming $\beta$ to $\tau$. None of these
left swaps involve the index $b$: Otherwise, by Lemma~\ref{lemma:lswapSeq}(III),
$b$ is the left index of such a swap, and some $t>b$ such that $\beta_t>\beta_b$ is the right index. 
This contradicts (\ref{eqn:June29bbb}) and (\ref{eqn:June29ccc}) combined. Thus
\begin{equation}
\label{eqn:June26ttt}
\tau_b = \beta_b.
\end{equation}

By Lemma~\ref{lemma:seg1big}, $\beta_{i_m(\beta)} \geq \beta_v$ for all $v > b$. 
That is $\max\{\beta_v:v>b\}\leq \beta_{i_m(\beta)}$. However, notice that 
$\{\beta_v:v>b\}=\{\tau_v:v>b\}$ since $\{\beta_v:v\leq b\}=\{\tau_v:v\leq b\}$. Therefore,
\begin{equation}
\label{June27rmax1}
\tau_{\rmax{b}{\tau}} \leq \beta_{i_m(\beta)}. 
\end{equation}

The assumption $b \in \seg{m}{2}(\beta)$ means $\beta_b < \beta_{i_{m-1}-1}$. Then the definition of $\beta$ and \eqref{eqn:June26ttt} imply $\tau_b < \tau_{i_{m-1}-1}$ (since $b>i_{m-1}-1\in \seg{m-1}{3}$). Hence, by \eqref{equation:rmin}, $\rmin{b}{\tau} = b$ and
\begin{equation}
\label{June27rmin1}
\tau_{\rmin{b}{\tau}} = \tau_b.
\end{equation}

\noindent
\emph{Case (A).1:} ($\tau_{\rmin{b}{\tau}} \geq \tau_{\rmax{b}{\tau}}$) By \eqref{equation:flex}, $\flex{b}{\tau} = 0$. Then the definition of $b$ implies 
\[\tau_1 + \cdots + \tau_b + \flex{b}{\tau} = \tau_1 + \cdots + \tau_b < \gamma_1 + \cdots + \gamma_b,\]
establishing (\ref{eqn:June29thegoal}).

\noindent
\emph{Case (A).2:} ($\tau_{\rmin{b}{\tau}} < \tau_{\rmax{b}{\tau}}$) Thus,
$\flex{b}{\tau} = \tau_{\rmax{b}{\tau}} - \tau_{\rmin{b}{\tau}} - 1$, and
\begin{equation}
\label{eqn:June26bigeqn}
\begin{array}{ll}
\tau_1 + \cdots + \tau_b + \flex{b}{\tau} & = \beta_1 + \cdots + \beta_b + \tau_{\rmax{b}{\tau}} - \tau_{\rmin{b}{\tau}} - 1 \\
&  \leq \beta_1 + \cdots + \beta_b + \beta_{i_m(\beta)} - \beta_b - 1 \\
& = \beta_1 + \cdots + \beta_{b-1} + \beta_{i_m(\beta)} - 1,
\end{array}
\end{equation}
where the first equality follows from \eqref{eqn:June26ttt}, and the inequality is by \eqref{June27rmax1} and \eqref{June27rmin1}. 

We have two subcases:

\noindent
\emph{Case (A).2.1} ($r=i_m(\beta)$) Then \eqref{eqn:June25zzz} and \eqref{eqn:June26bigeqn} imply \[\tau_1 + \cdots + \tau_b + \flex{b}{\tau} < \beta_1 + \cdots + \beta_{b-1} + \beta_{i_m(\beta)} = \gamma_1 + \cdots + \gamma_b,\] 
which proves (\ref{eqn:June29thegoal}).

\noindent
\emph{Case (A).2.2} ($r > i_m(\beta)$) By (\ref{eqn:June25zzz}),
$\beta_b<\beta_r$. Combining this with (\ref{eqn:June29yyy}), we may apply Lemma \ref{lemma:1032avoidingConsequence} to $\beta$ with $r,m$ being as above, and $s=b$, to conclude that $\beta_b=\beta_{i_m(\beta)-1}$ and $\beta_{i_m(\beta)} = \beta_r = \beta_b + 1$. Applying this to \eqref{eqn:June26bigeqn} yields
\begin{align*}
\tau_1 + \cdots + \tau_b + \flex{b}{\tau} & \leq \beta_1 + \cdots + \beta_{b-1} + \beta_{i_m(\beta)} - 1 \\
&  = \beta_1 + \cdots + \beta_{b-1} + (\beta_b + 1) - 1 \\
& = \beta_1 + \cdots + \beta_{b} \\
& < \gamma_1 + \cdots + \gamma_{b}, 
\end{align*}
where the final inequality follows from \eqref{eqn:June25zzz} and the definition of $\beta$; this again proves
(\ref{eqn:June29thegoal}).

\noindent
\emph{Case (B):} ($b\in\seg{m}{3}(\beta)$ (\emph{i.e.,} $b=i_m(\beta)-1$) and $r \in \seg{m+1}{1}(\beta)$) If $b < s \leq r$, then $s\in \seg{m+1}{1}(\beta)$ and thus $\beta_s \geq \beta_r$. For such an $s$,
if a left swap involving indices $b$ and $s$ occurred in the transformation from $\beta$ to $\tau$, then it follows from Lemma~\ref{lemma:lswapSeq}(III) that
\begin{equation}
\label{eqn:June29copy1}
\tau_b \geq \beta_s \geq \beta_r = \gamma_b.
\end{equation} 
This contradicts $b$'s definition. So such a left swap cannot exist. This, with Lemma~\ref{lemma:lswapSeq}(III)
means $s$ is not the right index of a left swap in the $\beta$ to $\tau$ transformation. Notice 
\[\beta_{i_{m}(\beta)-1}=\beta_b<\beta_r\leq \beta_s,\] 
where the middle inequality is by (\ref{eqn:June25zzz}). Thus, since $s\in\seg{m+1}{1}(\beta)$, by applying Lemma~\ref{lemma:seg1big} to $\beta$ with $i=s$ we conclude that
\begin{equation}
\label{eqn:June30mnb1}
\beta_s \geq \beta_t \text{ for all $t > s$.}
\end{equation}
Hence $s$ cannot be the leftmost index of a left swap that transforms
$\beta$ to $\tau$. 

All of the above analysis, together with the definition of $\beta$, shows that
\begin{equation}
\label{eqn:June26yyy}
\beta_s = \tau_s \text{ for $1 \leq s \leq r$ with $s \neq b$.}
\end{equation}

A similar argument proves
\begin{equation}
\label{eqn:June26xxx}
\beta_s = \gamma_s \text{ for $1 \leq s < r$ with $s \neq b$.}
\end{equation}
More precisely, if $b<s<r$ then $s\in \seg{m+1}{1}(\beta)$ and $\beta_s>\beta_r$ (by the minimality of $r$). From
this, we get a variation of (\ref{eqn:June29copy1}) which says $\gamma_b \geq \beta_s > \beta_r = \gamma_b$ (a contradiction). The remainder of the argument is the same.

Further, we claim:
\begin{equation}
\label{eqn:June27yyy}
\beta_b \leq \tau_b < \gamma_b=\beta_r\leq \beta_s= \tau_s \text{ for $b < s \leq r$.}
\end{equation}
The first inequality is by Lemma~\ref{lemma:lswapSeq}(III). The next inequality is from the definition of $b$.
The equality thereafter is (\ref{eqn:June25zzz}). The next inequality is since $s\leq r$ and $s,r\in \seg{m+1}{1}(\beta)$.
The remaining equality is (\ref{eqn:June26yyy}).

Now, \eqref{eqn:June27yyy} says, in particular, that $\tau_{b}<\tau_{b+1}$. Since $\tau\in {\overline {\sf KM}}_n^{\geq 1}$
(by the proof of Proposition~\ref{prop:exhaustiveQKMultFreeCases}), it avoids $(0,1,2)$. Hence $\tau_{b-1}\geq \tau_b$.
This combined with \eqref{eqn:June26yyy} implies 
\begin{equation}
\label{eqn:June30i_m}
b \in \seg{m}{3}(\tau) \text{ \ and $i_m(\beta)=i_m(\tau)=b+1$.}
\end{equation} 

If $r = i_m(\beta)$, then $r-1 = b$ (by this case's assumption). Since $b \in \seg{m}{3}(\tau)$, then $\beta_{b} \leq \beta_{i_{m-1}(\beta)-1}$ (otherwise $(\beta_{i_{m-1}(\beta)-1}, \beta_{b},\beta_{b+1})\simeq (0,1,2)$). Hence 
$\rmin{r-1}{\tau} = \rmin{b}{\tau} = b$. Otherwise, if $r > i_m(\beta)$, then \eqref{eqn:June26yyy} and \eqref{eqn:June27yyy} imply $\rmin{r-1}{\tau} = b$. Hence, 
\begin{equation}
\label{eqn:June27dsa}
\tau_{\rmin{r-1}{\tau}} = \tau_{b}
\end{equation}

If $r = i_m(\beta)$, then $r-1 = b$ (by this case's assumption). By definition \eqref{equation:rmax} and \eqref{eqn:June30i_m}, $\rmax{b}{\tau}=i_m(\tau)=i_m(\beta)=b+1=r$. Otherwise, $r > i_m(\beta)$. Now, by \eqref{eqn:June30mnb1}, $\beta_r \geq \beta_t$ for all $t > r$.
This, combined with \eqref{eqn:June26yyy} and \eqref{eqn:June27yyy}, implies $\tau_r \geq \tau_t$ for all $t > r$. 
This implies, by definition \eqref{equation:rmax}, that $\rmax{r-1}{\tau}=r$. This, combined with \eqref{eqn:June26yyy}, implies 
\begin{equation}
\label{eqn:June27ewq}
\tau_{\rmax{r-1}{\tau}} = \tau_{r} = \beta_r
\end{equation}

\noindent
\emph{Case (B).1:} ($\tau_{\rmin{r-1}{\tau}} \geq \tau_{\rmax{r-1}{\tau}}$) By \eqref{equation:flex}, $\flex{r-1}{\tau} = 0$. Then \eqref{eqn:June26yyy}, \eqref{eqn:June26xxx}, and the definition of $b$, imply 
\[\tau_1 + \cdots + \tau_{r-1} + \flex{r-1}{\tau} = \tau_1 + \cdots + \tau_{r-1} < \gamma_1 + \cdots + \gamma_{r-1}.\]

\noindent
\emph{Case (B).2:} ($\tau_{\rmin{r-1}{\tau}} < \tau_{\rmax{r-1}{\tau}}$) Here,
$\flex{r-1}{\tau} = \tau_{\rmax{r-1}{\tau}} - \tau_{\rmin{r-1}{\tau}} - 1$, and \eqref{eqn:June26yyy}, \eqref{eqn:June27dsa}, \eqref{eqn:June27ewq}, \eqref{eqn:June26xxx}, and \eqref{eqn:June25zzz} (in that order) give 
\vspace{-0.25in}\begin{center}
$\begin{array}{ll}
\tau_1 \!+\! \cdots \!+\! \tau_{r-1} \!+\! \flex{r-1}{\tau}\!\!\!\! & =  \beta_1 + \cdots + \beta_{b-1} + \tau_b + \beta_{b+1}  + \cdots + \beta_{r-1}  + \tau_{\rmax{b}{\tau}} \!-\! \tau_{\rmin{b}{\tau}} \!-\! 1 \\
& = \beta_1 + \cdots + \beta_{b-1} + \tau_b + \beta_{b+1}  + \cdots + \beta_{r-1} + \beta_r - \tau_b - 1 \\
& = \beta_1 + \cdots + \beta_{b-1}  + \beta_r + \beta_{b+1} + \cdots + \beta_{r-1} - 1 \\
& < \gamma_1 + \cdots + \gamma_{r-1}.
\end{array}$
\end{center}

\noindent
\emph{Case (C):} ($b\in \seg{m}{3}(\beta)$ (\emph{i.e.,} $b=i_m(\beta)-1$) and $r \geq i_{m+1}(\beta)$) Define
\begin{equation}
\label{eqn:June28ooo}
z = \max \{ v : v \in \seg{m+1}{}(\beta), \beta_v \geq \beta_{i_{m+1}(\beta)}, \text{ and } \beta_v > \beta_{i_{m}(\beta) - 1} \}.
\end{equation}
Note that $z\in [1,n]$ since the set is nonempty: it always contains $i_{m}(\beta)$. Clearly, $z < r$ since 
$r\geq i_{m+1}(\beta) \in \seg{m+2}{}$. Let $b < s \leq z < r$. Then $\beta_s \neq \beta_r$ by the minimality of $r$. By definition of $\seg{m+1}{}$ and of $z$, $\beta_s \geq \beta_z > \beta_{i_{m}(\beta) - 1}$. Thus if $\beta_s < \beta_r$, then $(\beta_{i_{m}(\beta) - 1}, \beta_s, \beta_r) \simeq (0,1,2)$. Hence $\beta_s > \beta_r$. For such an $s$,
if a left swap involving indices $b$ and $s$ occurred in the transformation from $\beta$ to $\tau$, then it follows from Lemma~\ref{lemma:lswapSeq}(III) that
\begin{equation}
\label{eqn:June29copy}
\tau_b \geq \beta_s > \beta_r = \gamma_b.
\end{equation} 
This contradicts $b$'s definition. Thus such a left swap cannot exist. This, with Lemma~\ref{lemma:lswapSeq}(III)
means $s$ is not the right index of any left swap in the transformation from
$\beta$ to $\tau$. Notice 
\[\beta_{i_{m}(\beta)-1}=\beta_b<\beta_r < \beta_s,\] 
where the middle inequality is by (\ref{eqn:June25zzz}). Thus, since $\beta_s\geq \beta_z>\beta_{i_m(\beta)-1}$ and
$s\in\seg{m+1}{}(\beta)$, we may apply Lemma~\ref{lemma:seg1big} to $\beta$ with $i=s$ to conclude that
\begin{equation}
\label{eqn:June30mnb}
\beta_s \geq \beta_t \text{ for all $t > s$.}
\end{equation}
Hence $s$ cannot be the leftmost index of a left swap that transforms
$\beta$ to $\tau$. 

All of the above analysis, together with the definition of $\beta$, shows that
\begin{equation}
\label{eqn:June28yyy}
\beta_s = \tau_s \text{ for $1 \leq s \leq z$ with $s \neq b$.}
\end{equation}

The same argument, replacing $\tau$ with $\gamma$ throughout proves
\begin{equation}
\label{eqn:June28xxx}
\beta_s = \gamma_s \text{ for $1 \leq s \leq z$ with $s \neq b$.}
\end{equation}

Further, we have the following inequality
\begin{equation}
\label{eqn:June28sas}
\beta_b \leq \tau_b < \gamma_b= \beta_r < \beta_s = \tau_s \text{ for $b < s \leq z < r$.}
\end{equation}
The first inequality is by Lemma~\ref{lemma:lswapSeq}(III). The next inequality is from the definition of $b$.
The equality thereafter is (\ref{eqn:June25zzz}). The next inequality is by the minimality of $r$.
The remaining equality is (\ref{eqn:June28yyy}).

Then \eqref{eqn:June28sas} implies $\tau_{i_m(\beta)-1}=\tau_{b} < \tau_z$. This, with \eqref{eqn:June28yyy} and \eqref{eqn:June28sas}, implies $\rmin{z}{\tau} = b$. Thus 
\begin{equation}
\label{eqn:June28rmin1}
\tau_{\rmin{z}{\tau}} = \tau_{b}.
\end{equation}

Let $z < t < i_{m+1}(\beta)$. First suppose $\beta_{i_{m+1}(\beta)} \leq \beta_t$. Then the maximality of $z$ implies $\beta_t \leq \beta_{i_{m}(\beta)-1}$. This implies $\beta_r \leq \beta_{i_{m+1}(\beta)} \leq \beta_t \leq \beta_{i_{m}(\beta)-1} = \beta_{b}$ (where the first inequality follows from Lemma~\ref{lemma:seg1big} applied to $i_{m+1}(\beta)$). This inequality contradicts \eqref{eqn:June25zzz}. Hence 
\begin{equation}
\label{eqn:June30qwerty}
\beta_{i_{m+1}(\beta)} > \beta_t.
\end{equation} 
Lemma~\ref{lemma:seg1big} applied to $i_{m+1}(\beta)$, implies $\beta_{i_{m+1}(\beta)} \geq \beta_v$ for all $v \geq i_{m+1}(\beta)$. Combining this with (\ref{eqn:June30qwerty}) shows that in fact 
\begin{equation}
\label{eqn:June30Dvorak}
\beta_{i_{m+1}(\beta)} \geq \beta_v, \text{\ for all $v >z$}.
\end{equation}

Now, by \eqref{eqn:June28yyy} and (\ref{eqn:June30Dvorak}), $\beta_{i_{m+1}(\beta)} \geq \tau_v, \text{\ for all $v >z$}$.
Since $\rmax{z}{\tau}>z$ (by definition),
\begin{equation}
\label{eqn:June28rmax1}
\tau_{\rmax{z}{\tau}} \leq \beta_{i_{m+1}(\beta)}.
\end{equation}

\noindent
\emph{Case (C).1:}  ($\tau_{\rmin{z}{\tau}} \geq \tau_{\rmax{z}{\tau}}$) By \eqref{equation:flex}, $\flex{z}{\tau} \!=\! 0$. Then \eqref{eqn:June28yyy}, \eqref{eqn:June28xxx}, and $b$'s definition imply 
\[\tau_1 + \cdots + \tau_{z} + \flex{z}{\tau} = \tau_1 + \cdots + \tau_{z} < \gamma_1 + \cdots + \gamma_{z}.\]

\noindent
\emph{Case (C).2:} ($\tau_{\rmin{z}{\tau}} < \tau_{\rmax{z}{\tau}}$)
Now, $\flex{z}{\tau} = \tau_{\rmax{z}{\tau}} - \tau_{\rmin{z}{\tau}} - 1$, and \eqref{eqn:June28yyy}, \eqref{eqn:June28rmin1}, \eqref{eqn:June28rmax1} shows
\begin{equation}\label{eqn:June28bigeqn2}
\begin{array}{ll}
\tau_1 \!+\! \cdots \!+\! \tau_{z} \!+\! \flex{z}{\tau}\!\!\!\!\! & = \beta_1 + \cdots + \beta_{b-1} + \tau_b + \beta_{b+1}  + \cdots + \beta_{z} + \tau_{\rmax{z}{\tau}} \!-\! \tau_{\rmin{z}{\tau}} - 1 \\
& \leq \beta_1 + \cdots + \beta_{b-1} + \tau_b + \beta_{b+1}  + \cdots + \beta_{z} + \beta_{i_{m+1}(\beta)} - \tau_b - 1 \\
& = \beta_1 + \cdots + \beta_{b-1}  + \beta_{i_{m+1}(\beta)} + \beta_{b+1} + \cdots + \beta_{z} - 1 \\
\end{array}
\end{equation}
We have two subcases:

\noindent
\emph{Case (C).2.1:} ($r=i_{m+1}(\beta)$) Then \eqref{eqn:June25zzz}, \eqref{eqn:June28xxx} and \eqref{eqn:June28bigeqn2} give \[\tau_1 + \cdots + \tau_z + \flex{z}{\tau} < \beta_1 + \cdots + \beta_{b-1}  + \beta_{i_{m+1}(\beta)} + \beta_{b+1} + \cdots + \beta_{z} = \gamma_1 + \cdots + \gamma_z.\] 

\noindent
\emph{Case (C).2.2:} ($r > i_{m+1}(\beta)$)   By (\ref{eqn:June29yyy}), we may apply Lemma \ref{lemma:1032avoidingConsequence} to $\beta$ with $r,m+1$ being as above, and $s=b$.
Since (\ref{eqn:June25zzz}) says $\beta_b<\beta_r$, said lemma shows
 $\beta_{i_{m+1}(\beta)} = \beta_r = \beta_b + 1$. Applying this to \eqref{eqn:June28bigeqn2} yields
\begin{align*}
\tau_1 + \cdots + \tau_z + \flex{z}{\tau} & \leq \beta_1 + \cdots + \beta_{b-1}  + \beta_{i_{m+1}(\beta)} + \beta_{b+1} + \cdots + \beta_{z} - 1 \\ 
&  = \beta_1 + \cdots + \beta_z \\ 
& < \gamma_1 + \cdots + \gamma_{z}, 
\end{align*}
where the final inequality follows from \eqref{eqn:June25zzz} and \eqref{eqn:June28xxx}.
\end{proof}

\begin{corollary}
\label{corollary:quasikeyDisjoint}
Let $\gamma,\tau \in {\sf lswap}(\alpha)$ with $\gamma >_{lex} \tau$. If $T \in {\sf qKT}(\gamma)$, $S \in {\sf qKT}(\tau)$, then ${\sf wt}(T) \neq {\sf wt}(S)$.
\end{corollary}
\begin{proof}
Suppose not, and $\mu = {\sf wt}(T) = {\sf wt}(S)$. Then the two parts of 
Lemma \ref{lemma:dominanceInterval} give 
$$\gamma_1+\cdots+\gamma_b \leq \mu_1+\cdots+\mu_b \leq \tau_1+\cdots+\tau_b+\flex{b}{\tau}$$ for all $b\in [1,n]$. This contradicts Proposition \ref{theorem:dominanceIntervalDisjoint}.
\end{proof}

\noindent
\emph{Conclusion of the proof of sufficiency}:
Define $\alpha'\in {\sf Comp}_n$ by $\alpha'_i = \alpha_i + 1$. Since $\alpha\in 
 {\overline {\sf KM}}_n$ then $\alpha'\in 
 {\overline {\sf KM}}_n^{\geq 1}$. Observe, 
$\kappa_{\alpha'}=x_1 \cdots x_n \cdot \kappa_{\alpha}$. Therefore
 $\kappa_{\alpha}$ is multiplicity-free if and only if $\kappa_{\alpha'}$ is multiplicity-free.  Now, $\kappa_{\alpha'}$ is the sum of $\mathfrak{D}_\beta$ for $\beta \in {\sf Qlswap}(\alpha')={\sf lswap}(\alpha')$.  
Each of these ${\mathfrak D}_{\beta}$'s are multiplicity-free by Theorem~\ref{thm:qksummary}.
Their sum is multiplicity-free by Corollary \ref{corollary:quasikeyDisjoint}. Hence $\kappa_{\alpha}$ is multiplicity-free.\qed

\section*{Acknowledgements}
We thank Mahir Can for helpful communications. We thank David Brewster 
and Husnain Raza for writing computer code (as part of their
NSF RTG funded ICLUE program) that was useful for checking parts of the proof.
We used the Maple packages ACE and Coxeter/Weyl in our investigations.
AY was partially supported by a Simons Collaboration Grant, and an NSF RTG grant.
RH was partially supported by an AMS-Simons Travel Grant.

\end{document}